\documentclass[draft]{amsart}
\usepackage{amssymb}
\usepackage{braket}
\usepackage[all]{xy}

\newtheorem{thm}{Theorem}[section]
\newtheorem{cor}[thm]{Corollary}
\newtheorem{prop}[thm]{Proposition}
\newtheorem{lem}[thm]{Lemma}
\newtheorem{claim}[thm]{Claim}
\theoremstyle{definition}
\newtheorem{dfn}[thm]{Definition}
\newtheorem{ex}[thm]{Example}
\newtheorem{fact}[thm]{Fact}

\theoremstyle{remark}
\newtheorem{rem}[thm]{Remark}
\newtheorem{caution}[thm]{Caution}

\newcommand{\Ob}{\mathrm{Ob}}
\newcommand{\ppr}{^{\prime}}
\newcommand{\GS}{{}_G\mathit{Set}}
\newcommand{\C}{{}_G\mathit{Set}_{\mathcal{C}}}
\newcommand{\CO}{{}_G\mathit{Set}_{\mathcal{C},\mathcal{O}}}
\newcommand{\COC}{{}_G\mathit{Set}_{\mathcal{C},\mathcal{O}_{\mathcal{C}}}}
\newcommand{\COM}{{}_G\mathit{Set}_{\mathcal{C},\mathcal{O}_{\bullet}}}

\newcommand{\HG}{\text{-}{}_G\mathit{Set}_{\mathcal{C},\mathcal{O}_{\mathcal{C}}}}
\newcommand{\TS}{(\mathcal{C},\mathcal{O}_{\mathcal{C}},\mathcal{O}_{\bullet})}
\newcommand{\Tam}{\mathit{Tam}_{(\mathcal{C},\mathcal{O}_{\bullet})}}
\newcommand{\STam}{\mathit{STam}_{(\mathcal{C},\mathcal{O}_{\bullet})}}
\newcommand{\SMack}{\mathit{SMack}_{(\mathcal{C},\mathcal{O}_{\bullet})}}
\newcommand{\Mack}{\mathit{Mack}_{(\mathcal{C},\mathcal{O}_{\bullet})}}

\newcommand{\Mon}{\mathit{Mon}}
\newcommand{\Ab}{\mathit{Ab}}
\newcommand{\SRing}{\mathit{SRing}}
\newcommand{\Ring}{\mathit{Ring}}
\newcommand{\triv}{_\mathrm{triv}}
\newcommand{\Br}{\mathit{Br}}

\newcommand{\MaddG}{\mathit{Madd}(G)}
\newcommand{\RaddG}{\mathit{Radd}(G)}
\newcommand{\TamG}{\mathit{Tam}(G)}

\newcommand{\SMackG}{\mathit{SMack}(G)}
\newcommand{\GreenG}{\mathit{Green}(G)}
\newcommand{\sG}{{}_G\mathit{set}}
\newcommand{\ev}{\mathit{ev}^G}
\newcommand{\CrossG}{G\text{-}\mathit{xset}/QX}
\newcommand{\GRing}{G\text{-}\mathit{Ring}}
\newcommand{\ResalgG}{\mathit{Res}_{\mathrm{alg}}(G)}

\newcommand{\ApX}{(A\overset{p}{\rightarrow}X)}
\newcommand{\Am}{(A\overset{p}{\rightarrow}X,m_A)}
\newcommand{\Amo}{(A_1\overset{p_1}{\rightarrow}X,m_{A_1})}
\newcommand{\Amt}{(A_2\overset{p_2}{\rightarrow}X,m_{A_2})}
\newcommand{\Ami}{(A_i\overset{p_i}{\rightarrow}X,m_{A_i})}

\newtheorem{thmd}{{\rm \textbf{Theorem \ref{Thm1}}}$^{\prime}$}

\newtheorem{thmf}{{\rm \textbf{Theorem \ref{Thm1}}}}

\newtheorem{thms}{{\rm \textbf{Theorem \ref{Thm2}}}}

\newtheorem{thmt}{{\rm \textbf{Theorem \ref{Thm3}}}}

\newtheorem{factBrun}{{\rm \textbf{Fact \ref{FactBrun}}}}

\numberwithin{equation}{section}

\begin{document}

\title[Tambarization of a Mackey functor]{Tambarization of a Mackey functor and its application to the Witt-Burnside construction}

\author{Hiroyuki NAKAOKA}
\address{Department of Mathematics and Computer Science, Kagoshima University, 1-21-35 Korimoto, Kagoshima, 890-0065 Japan}

\email{nakaoka@sci.kagoshima-u.ac.jp}

\thanks{The author wishes to thank Professor Toshiyuki Katsura for his encouragement}
\thanks{The author wishes to thank Professor Fumihito Oda for stimulating arguments and useful comments, Professor Daisuke Tambara and Professor Tomoyuki Yoshida for their comments and advices}
\thanks{Supported by JSPS Grant-in-Aid for Young Scientists (B) 22740005}

\begin{abstract}
For an arbitrary group $G$, a ({\it semi}-){\it Mackey functor} is a pair of covariant and contravariant functors from the category of $G$-sets, and is regarded as a $G$-bivariant analog of a commutative (semi-)group. In this view, a $G$-bivariant analog of a (semi-)ring should be a (semi-)Tambara functor. A Tambara functor is firstly defined by Tambara, which he called a TNR-functor, when $G$ is finite. As shown by Brun, a Tambara functor plays a natural role in the Witt-Burnside construction.

It will be a natural question if there exist sufficiently many examples of Tambara functors, compared to the wide range of Mackey functors. In the first part of this article, we give a general construction of a Tambara functor from any Mackey functor, on an arbitrary group $G$. In fact, we construct a functor from the category of semi-Mackey functors to the category of Tambara functors.
This functor gives a left adjoint to the forgetful functor, and can be regarded as a $G$-bivariant analog of the monoid-ring functor.

In the latter part, when $G$ is finite, we invsetigate relations with other Mackey-functorial constructions ---crossed Burnside ring, Elliott's ring of $G$-strings, Jacobson's $F$-Burnside ring--- all these lead to the study of the Witt-Burnside construction.
\end{abstract}

\maketitle

\section{Introduction and Preliminaries}
For an arbitrary group $G$, a (semi-)Mackey functor is a pair of covariant and contravariant functors from the category of \lq $G$-sets'. More precisely, as defined in \cite{B-B}, if we are given a {\it Mackey system} $(\mathcal{C},\mathcal{O})$ on $G$, then subcategories of the category of $G$-sets $\C$ and $\CO$ are associated, and a {\it Mackey functor} is a pair $M=(M^{\ast},M_{\ast})$ of a covariant functor $M_{\ast}\colon\CO\rightarrow \Ab$ and an additive contravariant functor $M^{\ast}\colon\C\rightarrow \Ab$. If $G$ is finite, we obtain an ordinary (semi-)Mackey functor on $G$ as in Remark \ref{RemFinMack}.

As in \cite{Yoshida}, for any finite group $G$, a (semi-)Mackey functor can be regarded as a $G$-bivariant analog of a commutative (semi-)group. In fact if $G$ is trivial, the category of (semi-)Mackey functors are canonically equivalent to the category of commutative (semi-)groups.
In this view, a $G$-bivariant analog of a (semi-)ring will be a (semi-)Tambara functor. A Tambara functor is firstly defined by Tambara in \cite{Tam}, which he called a TNR-functor, for any finite group $G$. As shown by Brun, a Tambara functor plays a natural role in the Witt-Burnside construction (\cite{Brun}). (See Fact \ref{FactBrun} in this article.)

It will be a natural question if there exist sufficiently many examples of Tambara functors, compared to the wide range of Mackey functors. In this article, we give a general construction of a Tambara functor from any Mackey functor, on an arbitrary group $G$. In fact, we construct a functor from the category of semi-Mackey functors to the category of semi-Tambara functors.
More precisely, in Theorem \ref{Thm1}, for any Tambara system (Definition \ref{DefTamSys}) $\TS$ on $G$, we construct a functor
\[ \mathcal{S}\colon\SMack\rightarrow\STam, \]
where $\SMack$ and $\STam$ are appropriately defined categories of semi-Mackey and semi-Tambara functors, respectively: 
\begin{thmf}
Let $\TS$ be a Tambara system on $G$. The construction of $\mathcal{S}_M$ in Proposition \ref{PropForThm1} gives a functor
\[ \mathcal{S}\colon\SMack\rightarrow\STam. \]
\end{thmf}
\noindent Composing $\mathcal{S}$ with other known functors, we obtain functors
\begin{eqnarray*}
\mathcal{T}\colon\SMack&\rightarrow&\Tam,\\
\Mack&\rightarrow&\STam,\\
\Mack&\rightarrow&\Tam,
\end{eqnarray*}
where $\Mack$ and $\Tam$ are the categories of Mackey and Tambara functors, respectively.

Moreover in Theorem \ref{Thm2}, we show $\mathcal{S}$ is left adjoint to the forgetful functor $\STam\rightarrow\SMack$.
\begin{thms}
Let $\TS$ be any Tambara system on $G$.
The functor constructed in Theorem \ref{Thm1}
\[ \mathcal{S}\colon\SMack\rightarrow\STam \]
is left adjoint to the forgetful functor $\mu\colon\STam\rightarrow\SMack$.
\end{thms}
\noindent As a corollary, composition
\[ \mathcal{T}=\gamma\circ\mathcal{S}\colon\SMack\rightarrow\Tam \]
is left adjoint to the forgetful functor $\mu\circ \mathcal{U}\colon\Tam\rightarrow\SMack$ (Corollary \ref{CorOfThm2}). Thus Theorem \ref{Thm2} means that $\mathcal{T}$ can be regarded as a $G$-bivariant analog of the monoid-ring functor.

\bigskip

In the latter part of this article, we restrict ourselves to the finite group case, and investigate the connection to other constructions, especially in relation with the Witt-Burnside construction.

Firstly in section \ref{3-1}, we show the relation with the crossed Burnside ring. In fact, if we are given a $G$-monoid $Q$, then by applying Corollary \ref{CorOfThm2} to the fixed point functor $\mathcal{P}_Q$  associated to $Q$, we obtain a Tambara functor $\mathcal{T}_{\mathcal{P}_Q}$. This generalizes the crossed Burnside ring functor in \cite{O-Y2}, \cite{O-Y3}. Indeed, if $Q$ is a finite $G$-monoid, we have an isomorphism of Tambara functors between $\mathcal{T}_{\mathcal{P}_Q}$ and the crossed Burnside ring functor $\Omega_Q$:
\[ \mathcal{T}_{\mathcal{P}_Q}\cong\Omega_Q \]

Secondly in section \ref{3-2}, we show the relationship with the Witt-Burnside construction. Historically, Tambara functors were applied to the Witt-Burnside construction firstly by Brun:
\begin{factBrun}[Theorem B, Theorem 15 in \cite{Brun}]
For any finite group $G$, the evaluation at $G/e$
\[ \TamG\rightarrow\GRing\ ; \ T\mapsto T(G/e) \]
has a right adjoint functor $L_G$. Here, $\GRing$ denotes the category of $G$-rings. If $G$ acts trivially on a ring $R$, then for any subgroup $H\le G$, there is an isomorphism
\[ \mathbb{W}_H(R)\cong L_G(R)(G/H). \]
\end{factBrun}

\noindent Here $\mathbb{W}_H(R)$ is the Witt-Burnside ring associated to $H$ and $R$.
Combining this with Theorem \ref{Thm2}, we obtain an isomorphism of functors, which gives a description of the Witt-Burnside ring of a monoid-ring:
\begin{thmt}
For any finite group $G$, there is an isomorphism of functors
\begin{equation*}
\mathbb{W}\circ\mathbb{Z}[-]\cong\mathcal{T}\circ\mathcal{L}.
\qquad
\xy
(-12,6)*+{\Mon}="0";
(12,6)*+{\SMackG}="2";
(-12,-6)*+{\Ring}="4";
(12,-6)*+{\TamG}="6";
{\ar^<<<<<<{\mathcal{L}} "0";"2"};
{\ar_{\mathbb{Z}[-]} "0";"4"};
{\ar^{\mathcal{T}} "2";"6"};
{\ar_<<<<<<{\mathbb{W}} "4";"6"};
{\ar@{}|\circlearrowright "0";"6"};
\endxy
\end{equation*}
\end{thmt}

\noindent Here, $\mathcal{L}$ is the functor defined in Claim \ref{ClaimConst}. 
Theorem \ref{Thm3} can be also regarded as an enhancement of the Elliott's isomorphism of rings (Theorem 1.7 in \cite{Elliott}) to a functorial level.

Thirdly in section \ref{3-3}, relating our construction to Jacobson's $F$-Burnside ring functor, we investigate the underlying Green functor structure of $\mathcal{T}_M$.
Jacobson functorially associated a Green functor $\mathcal{A}_F$ to any additive contravariant functor $F$ (\cite{Jacobson}). In fact, this functor gives the underlying Green functor structure of $\mathcal{T}_M$. Namely, we have an isomorphism of Green functors $\mathcal{T}_M\cong\mathcal{A}_{M^{\ast}}$. In this view, Theorem \ref{Thm1} can be restated as follows:
\begin{thmd}
Let $G$ be a finite group and $F$ be an object in $\MaddG$. 
If $F$ is moreover a semi-Mackey functor, i.e., if there exists a semi-Mackey functor with the contravariant part $F$, then $\mathcal{A}_F$ has a structure of a Tambara functor.
\end{thmd}
In turn, as an application of Theorem \ref{Thm1}$\ppr$, we can endow a Tambara functor structure on the Brauer ring functor defined by Jacobson (Corollary \ref{CorBr}).

Moreover, combining this with Theorem 3.13 in \cite{N_BrRng}, the underlying Green functor structure is also written by using Boltje's $(-)_+$-construction. Namely, we have an isomorphism of Green functors $\mathcal{T}_M\cong(\mathcal{R}_{\mathbb{Z}[M^{\ast}]})_+$ (Corollary \ref{CorPlus}).

\bigskip

\bigskip

Throughout this article, we fix an arbitrary group $G$. Its unit element is denoted by $e$.
$H\le G$ means $H$ is a subgroup of $G$.
$\GS$ denotes the category of (not necessarily finite) $G$-sets and $G$-equivariant maps. The category of {\it finite} $G$-sets is denoted by $\sG$.
For a $G$-set $X$ and a point $x\in X$, we denote the stabilizer of $x$ by $G_x$. For any $H\le G$ and $g\in G$, ${}^g\! H$ and $H^g$ denotes the conjugation ${}^g\! H=gHg^{-1}$, $H^g=g^{-1}Hg$.
Monoids are assumed to be commutative and have multiplicative units $1$. (Semi-)rings are assumed to be commutative, and have additive units $0$ and multiplicative units $1$.
We denote the category of monoids by $\Mon$, the category of (resp. semi-)rings by $\Ring$ (resp. $\SRing$), and the category of abelian groups by $\Ab$. A monoid homomorphism preserves $1$, and a (semi-)ring homomorphism preserves $0$ and $1$.

For any category $\mathcal{K}$ and any pair of objects $X$ and $Y$ in $\mathcal{K}$, the set of morphisms from $X$ to $Y$ in $\mathcal{K}$ is denoted by $\mathcal{K}(X,Y)$. For each $X\in\Ob(\mathcal{K})$, the comma category of $\mathcal{K}$ over $X$ is denoted by $\mathcal{K}/X$.

\begin{dfn}[\cite{B-B}]\label{DefMackSys}
A {\it Mackey system} on $G$ is a pair $(\mathcal{C},\mathcal{O})$ of
\begin{enumerate}
\item[{\rm (a)}] a set $\mathcal{C}$ of subgroups of $G$, closed under conjugation and finite intersections,
\item[{\rm (b)}] a family $\mathcal{O}=\{\mathcal{O}(H)\}_{H\in\mathcal{C}}$ of subsets
\[ \mathcal{O}(H)\subseteq\mathcal{C}(H), \]
where $\mathcal{C}(H)$ is defined by $\mathcal{C}(H)=\{ U\in\mathcal{C}\mid U\le H\}$ for each $H\in\mathcal{C}$,
\end{enumerate}
which satisfies
\begin{enumerate}
\item[{\rm (i)}] $[H:U]<\infty$
\item[{\rm (ii)}] $\mathcal{O}(U)\subseteq\mathcal{O}(H)$
\item[{\rm (iii)}] $\mathcal{O}(gHg^{-1})=g\mathcal{O}(H)g^{-1}$
\item[{\rm (iv)}] $U\cap V\in\mathcal{O}(V)$
\end{enumerate}
for all $H\in\mathcal{C}$, $U\in\mathcal{O}(H)$, $V\in\mathcal{C}(H)$ and $g\in G$.
\end{dfn}

\begin{ex}\label{ExMackSys}
Let $\mathcal{C}$ be a set of subgroups as in {\rm (a)} in Definition \ref{DefMackSys}. If we define $\mathcal{O}_{\mathcal{C}}=\{\mathcal{O}_{\mathcal{C}}(H)\}_{H\in\mathcal{C}}$ by
\[ \mathcal{O}_{\mathcal{C}}(H)=\{ U\in\mathcal{C}(H)\mid [H:U]<\infty \} \]
for each $H\in\mathcal{C}$, then $(\mathcal{C},\mathcal{O}_{\mathcal{C}})$ is a Mackey system.

In particular, if $G$ is a topological group and $\mathcal{C}$ is the set of all closed (resp. open) subgroups of $G$, we call $(\mathcal{C},\mathcal{O}_{\mathcal{C}})$ the {\it natural} (resp. {\it open-natural}) Mackey system on $G$ (Definition 2.4 in \cite{N_Tam}).
\end{ex}

\begin{dfn}\label{DefGSets}
For a Mackey system $(\mathcal{C},\mathcal{O})$ on $G$, define subcategories $\CO\subseteq \C$ of $\GS$ as follows.
\begin{enumerate}
\item $\C$ is a full subcategory of $\GS$, whose objects are those $X\in\Ob(\GS)$ satisfying $G_x\in\mathcal{C}$ for any $x\in X$.
\item $\CO$ is a category with the same objects as $\C$, whose morphisms from $X$ to $Y$ in $\CO$ are those $f\in \C(X,Y)$ satisfying the following.
\begin{enumerate}
\item[{\rm (i)}] $f$ has finite fibers, i.e., $f^{-1}(y)$ is a finite set for each $y\in Y$.
\item[{\rm (ii)}] $G_x\in\mathcal{O}(G_{f(x)})$ for each $x\in X$.
\end{enumerate}
\end{enumerate}
\end{dfn}

\begin{rem}
It can be easily seen that morphisms in $\CO$ are stable under pull-backs in $\C$. Namely, for any pull-back diagram in $\C$
\[
\xy
(-6,6)*+{X^{\prime}}="0";
(6,6)*+{Y^{\prime}}="2";
(-6,-6)*+{X}="4";
(6,-6)*+{Y}="6";
(8,-7)*+{,}="10";
(0,0)*+{\square}="8";
{\ar^{g^{\prime}} "0";"2"};
{\ar_{f^{\prime}} "0";"4"};
{\ar^{f} "2";"6"};
{\ar_{g} "4";"6"};
\endxy
\]
$g\in\CO(X,Y)$ implies $g\ppr\in \CO(X\ppr,Y\ppr)$.
\end{rem}

\begin{dfn}\label{DefMackFtr}
A {\it semi-Mackey functor} $M$ is a pair $(M_{\ast},M^{\ast})$ of
\begin{itemize}
\item[{\rm (a)}] contravariant functor $M^{\ast}\colon\C\rightarrow\Mon$,
\item[{\rm (b)}] covariant functor $M_{\ast}\colon\CO\rightarrow\Mon$,
\end{itemize}
satisfying the following conditions.
\begin{enumerate}
\item[{\rm (M0)}] $M^{\ast}(X)=M_{\ast}(X)$ for each $X\in\Ob(\C)$. We put $M(X)=M^{\ast}(X)=M_{\ast}(X)$.
\item[{\rm (M1)}] {\rm (Mackey condition)} If
\[
\xy
(-6,6)*+{X^{\prime}}="0";
(6,6)*+{Y^{\prime}}="2";
(-6,-6)*+{X}="4";
(6,-6)*+{Y}="6";
(0,0)*+{\square}="8";
{\ar^{g^{\prime}} "0";"2"};
{\ar_{f^{\prime}} "0";"4"};
{\ar^{f} "2";"6"};
{\ar_{g} "4";"6"};
\endxy
\]
is a pull-back diagram in $\C$ with $g\in\CO(X,Y)$, then
\[ M^{\ast}(f)\circ M_{\ast}(g)=M_{\ast}(g\ppr)\circ M^{\ast}(f\ppr). \]
\item[{\rm (M2)}] {\rm (Additivity)} For any direct sum decomposition $X=\displaystyle{\coprod_{\lambda\in\Lambda}}X_{\lambda}$ in $\C$, the natural map
\[ (M^{\ast}(i_{\lambda}))_{\lambda\in\Lambda}\colon M(X)\rightarrow \prod_{\lambda\in\Lambda}M(X_{\lambda}) \]
is an isomorphism, where $i_{\lambda}\colon X_{\lambda}\hookrightarrow X\ (\lambda\in\Lambda)$ are inclusions. In particular, $M(\emptyset)$ is trivial.
\end{enumerate}

A semi-Mackey functor $M$ is a {\it Mackey functor} if $M(X)$ is an abelian group for each $X\in\Ob(\C)$, namely if $M^{\ast}$ and $M_{\ast}$ are functors to $\Ab$. A {\it morphism of} ({\it semi}-){\it Mackey functors} from $M$ to $N$ is a family of monoid homomorphisms $\varphi=\{\varphi_X\colon M(X)\rightarrow N(X) \}_{X\in\Ob(\C)}$, which is natural with respect to both the contravariant part and the covariant part. The category of semi-Mackey functors (resp. Mackey functors) on $(\mathcal{C},\mathcal{O})$ is denoted by $\mathit{SMack}_{(\mathcal{C},\mathcal{O})}$ (resp. $\mathit{Mack}_{(\mathcal{C},\mathcal{O})}$).
\end{dfn}

\begin{rem}\label{RemFinMack}\label{FinInf}
Originally in \cite{B-B}, when $G$ is a profinite group, the open-natural Mackey system is called the {\it finite natural} Mackey system. A Mackey functor on a finite natural Mackey system is also called a $G$-{\it modulation}.

In particular when $G$ is a finite group, regarded as a discrete topological group, then the natural and the finite natural Mackey systems coincide, and we abbreviate the category of semi-Mackey functors on the (finite) natural Mackey system on $G$ to $\mathit{SMack}(G)$. 

In the ordinary definition of a Mackey (or a Tambara) functor, we work on the category of {\it finite} $G$-sets, instead of $\GS$.
However, by the additivity of a Mackey (or a Tambara) functor, the resulting categories become equivalent.
Of course, if we restrict ourselves to finite groups, then the same arguments as below are also possible using the category of finite $G$-sets.

Also remark that if $G=\{e\}$ is trivial, the evaluation at the trivial one-point $G$-set $\{e\}$ gives an equivalence of categories
\[ \mathit{SMack}(\{e\})\overset{\simeq}{\longrightarrow}\Mon\ \  ; \ \ M\mapsto M(\{e\}). \]
Similarly, $\mathit{Mack}(\{e\})$ is equivalent to $\Ab$.
\end{rem}

\begin{ex}\label{ExFix}
Let $(\mathcal{C},\mathcal{O})$ be any Mackey system on $G$.
For any $G$-monoid $Q$, there exists a naturally associated semi-Mackey functor $\mathcal{P}_Q\in\Ob(\mathit{Mack}_{(\mathcal{C},\mathcal{O})})$, which we call the {\it fixed point functor} valued in $Q$, as follows.
\begin{itemize}
\item[{\rm (a)}] For each $X\in\Ob(\C)$, define $\mathcal{P}_Q$ by 
\[ \mathcal{P}_Q(X)=\GS(X,Q). \]
\item[{\rm (b)}] For each morphism $f\in\C(Y,X)$,
\[ \mathcal{P}_Q^{\ast}(f)\colon \GS(X,Q)\rightarrow \GS(Y,Q) \]
is defined as the composition by $f$.
\item[{\rm (b)}] For each morphism $g\in\CO(Y,X)$,
\[ \mathcal{P}_{Q\,\ast}(g)\colon \GS(X,Q)\rightarrow \GS(Y,Q) \]
is defined by 
\[ (\mathcal{P}_{Q\,\ast}(g)(\alpha))(y)=\prod_{x\in g^{-1}(y)}\alpha(x), \]
for any $\alpha\in\GS(X,Q)$ and $y\in Y$.
\end{itemize}
\end{ex}

\section{Tambarization of a Mackey functor}

\subsection{Review on Tambara functors}

A Tambara system $(\mathcal{C},\mathcal{O}_+,\mathcal{O}_{\bullet})$ on $G$ consists of two Mackey systems $(\mathcal{C},\mathcal{O}_+)$ and $(\mathcal{C},\mathcal{O}_{\bullet})$ satisfying some compatibility conditions for exponential diagrams (see \cite{N_Tam}). In this article, we only consider the following special case (Proposition 4.5 in \cite{N_Tam}). 
As in the previous section, if one only consider the finite group case, then one may work over the category of finite $G$-sets, using the original definition by Tambara \cite{Tam}.

\begin{dfn}\label{DefTamSys}
Let $\widetilde{\mathcal{C}}$ be a set of subsets of $G$, closed under left and right translation, finite intersections and finite unions.
Put $\mathcal{C}=\{ H\in\widetilde{\mathcal{C}}\mid H\le G\}$. As in Example \ref{ExMackSys}, $(\mathcal{C},\mathcal{O}_{\mathcal{C}})$ is a Mackey system. We say $(\mathcal{C},\mathcal{O}_{\mathcal{C}},\mathcal{O}_{\bullet})$ is a {\it Tambara system} on $G$ if it satisfies the following.
\begin{enumerate}
\item[{\rm (i)}] $(\mathcal{C},\mathcal{O}_{\bullet})$ is a Mackey system on $G$.
\item[{\rm (ii)}] $\mathcal{O}_{\bullet}$ satisfies $H\in\mathcal{O}_{\bullet}(H)$ for each $H\in\mathcal{C}$.
\end{enumerate}
\end{dfn}

\begin{rem}
For any $X,Y\in\Ob(\C)$, we have
\[ \COC(X,Y)=\{ f\in\C(X,Y)\mid f\ \text{has finite fibers} \}. \]
Thus, for any commutative diagram in $\C$,
\[
\xy
(-8,6)*+{A}="0";
(8,6)*+{A\ppr}="2";
(0,-6)*+{X}="4";
(0,9)*+{}="6";
{\ar^{f} "0";"2"};
{\ar_{p} "0";"4"};
{\ar^{p\ppr} "2";"4"};
{\ar@{}|\circlearrowright "4";"6"};
\endxy
\]
$p\in\COC(A,X)$ implies $f\in\COC(A,A\ppr)$.
\end{rem}
By the above remark, category $\mathcal{S}_{\mathcal{C},\mathcal{O}|X}$ used in \cite{N_Tam} agrees with the comma category $\COC/X$, and we have the following (Definition 4.1 and Proposition 4.5 in \cite{N_Tam}).
\begin{rem}\label{RemExp}
Let $\TS$ be a Tambara system. Let $\eta\in\COM(X,Y)$ be any morphism.
For any $\ApX\in\Ob(\COC/X)$, we define $(\Pi_{\eta}(A)\overset{\pi=\pi(p)}{\longrightarrow}Y)$ in $\Ob(\COC/Y)$ by

\[\Pi_{\eta}(A)=\Set{(y,\sigma)|%
\begin{array}{l}%
y\in Y, \\
\sigma\colon \eta^{-1}(y)\rightarrow A \ \, \text{is a map of sets},\\
\ p\circ \sigma=\mathrm{id}_{\eta^{-1}(y)}%
\end{array}},
\quad
\pi(y,\sigma)=y. \]
We abbreviately write $\Pi_{\eta}\ApX=(\Pi_{\eta}(A)\overset{\pi}{\rightarrow} Y)$. For any morphism $a\colon \ApX\rightarrow(A\ppr\overset{p\ppr}{\rightarrow}X)$ in $\COC/X$, define $\Pi_{\eta}(a)\colon\Pi_{\eta}\ApX\rightarrow\Pi_{\eta\ppr}(A\ppr\overset{p\ppr}{\rightarrow}X)$ by $\Pi_{\eta}(a)(y,\sigma)=(y,a\circ\sigma)$.
Then we obtain a functor
\[ \Pi_{\eta}\colon\COC/X\rightarrow\COC/Y, \]
which is right adjoint to the pull-back functor defined by $\eta$
\[ X\underset{Y}{\times}-\colon\COC/Y\rightarrow\COC/X. \]
\end{rem}

By this adjoint property, for any $p\in \COC(A,X)$, we have a commutative diagram
\begin{equation}\label{DiagExp1}
\xy
(-14,6)*+{X}="0";
(-14,-6)*+{Y}="2";
(-1,6)*+{A}="4";
(9,6)*+{}="5";
(18,4.7)*+{X\underset{Y}{\times}\Pi_{\eta}(A)}="6";
(16,2)*+{}="7";
(16,-6)*+{\Pi_{\eta}(A)}="8";
{\ar_{\eta} "0";"2"};
{\ar_{p} "4";"0"};
{\ar_>>>>>{\lambda} "5";"4"};
{\ar^>>>>{\rho} "7";"8"};
{\ar^{\pi=\pi(p)} "8";"2"};
{\ar@{}|\circlearrowright "0";"8"};
\endxy
\end{equation}
where $\rho$ is the pull-back of $\eta$ by $\pi$, and $\lambda$ is the morphism corresponding to $\mathrm{id}_{\Pi_{\eta}(A)}$, and $p\circ\lambda$ becomes the pull-back of $\pi$ by $\eta$.
Any commutative diagram in $\C$
\begin{equation}\label{DiagExp2}
\xy
(-12,6)*+{X}="0";
(-12,-6)*+{Y}="2";
(0,6)*+{Z}="4";
(12,6)*+{X^{\prime}}="6";
(12,-6)*+{Y^{\prime}}="8";
{\ar_{\eta} "0";"2"};
{\ar_{\xi} "4";"0"};
{\ar_{\zeta} "6";"4"};
{\ar^{\eta\ppr} "6";"8"};
{\ar^{\upsilon} "8";"2"};
{\ar@{}|\circlearrowright "0";"8"};
\endxy
\end{equation}
isomorphic to $(\ref{DiagExp1})$ for some $p\in\COC(A,X)$, is called an {\it exponential diagram}.

\begin{dfn}[Definition 4.10, Remark 4.11 in \cite{N_Tam}]\label{DefTamFtr}
Let $\TS$ be a Tambara system on $G$. A {\it semi-Tambara functor} $S$ on $\TS$ is a triplet $(S^{\ast},S_+,S_{\bullet})$ consisting of
\begin{enumerate}
\item[{\rm (a)}] a semi-Mackey functor $(S^{\ast},S_+)$ on $(\mathcal{C},\mathcal{O}_{\mathcal{C}})$
\item[{\rm (b)}] a semi-Mackey functor $(S^{\ast},S_{\bullet})$ on $(\mathcal{C},\mathcal{O}_{\bullet})$
\end{enumerate}
which satisfies the following conditions.
\begin{enumerate}
\item[{\rm (i)}] $S(X)(=S^{\ast}(X)=S_+(X)=S_{\bullet}(X))$ is a semi-ring for each $X\in\Ob(\C)$, with the addition and the multiplication induced from the monoid structures of $S^{\ast}(X)=S_+(X)$ and $S^{\ast}(X)=S_{\bullet}(X)$, respectively.
\item[{\rm (ii)}] For any exponential diagram $(\ref{DiagExp2})$, we have
\[ S_+(\upsilon)\circ S_{\bullet}(\eta\ppr)\circ S^{\ast}(\zeta)=S_{\bullet}(\eta)\circ S_+(\xi). \]
\end{enumerate}

For any morphism $f$ in $\C$ (resp. $\COC$, $\COM$), we abbreviate $S^{\ast}(f)$ (resp. $S_+(f)$, $S_{\bullet}(f)$) to $f^{\ast}$ (resp. $f_+$, $f_{\bullet}$). These are called the {\it structure morphisms} of semi-Tambara functor $S$. $f_+$ (resp. $f_{\bullet}$) is called an {\it additive} (resp. {\it multiplicative}) {\it transfer}.

If $S(X)$ is moreover a ring for each $X\in\Ob(\C)$, then $S=(S^{\ast},S_+,S_{\bullet})$ is called a {\it Tambara functor}.
\end{dfn}

If $S$ and $T$ are (semi-)Tambara functors, a {\it morphism} $\psi$ from $S$ to $T$ is defined to be a family $\psi=\{\psi_X\}_{X\in\Ob(\C)}$ of semi-ring homomorphisms $\psi_X\colon S(X)\rightarrow T(X)$, compatible with structure morphisms. In other words, $\psi$ is a morphism of (semi-)Tambara functors if and only if it gives morphisms of additive and multiplicative semi-Mackey functors
\[ \psi\colon(S^{\ast},S_+)\rightarrow(T^{\ast},T_+),\quad\psi\colon(S^{\ast},S_{\bullet})\rightarrow(T^{\ast},T_{\bullet}). \]
We denote the category of semi-Tambara functors on $\TS$ by $\STam$. (We omitted $\mathcal{O}_{\mathcal{C}}$, since it is determined by $\mathcal{C}$.)

\begin{rem}[Theorem 5.16 and Proposition 5.17 in \cite{N_Tam}]\label{RemForget}
By definition, if we denote the multiplicative part by $S^{\mu}=(S^{\ast},S_{\bullet})$ for each $S\in\Ob(\STam)$, this gives a forgetful functor
\begin{eqnarray*}
\mu\colon\STam&\rightarrow&\SMack\\
(\psi\colon S\rightarrow T)&\mapsto&(\psi\colon S^{\mu}\rightarrow T^{\mu}).
\end{eqnarray*}

The category of Tambara functors is denoted by $\Tam$. This is a full subcategory in $\STam$, whose inclusion we denote by
\[ \mathcal{U}\colon\Tam\rightarrow\STam. \]
Conversely, if we are given an object $S$ in $\STam$, then the correspondence
\[ \gamma S\colon\Ob(\C)\ni X\mapsto K_0(S(X))\in\Ob(\Ring) \]
becomes a Tambara functor, with appropriately defined structure morphisms. Here, $K_0\colon\SRing\rightarrow\Ring$ is the ring-completion functor.

Moreover, for any $\psi\in\STam(S,T)$, $\gamma\psi:=\{ K_0(\psi_X)\}_{X\in\Ob(\C)}$ gives a morphism of Tambara functors $\gamma\psi:\gamma S\rightarrow\gamma T$, and we obtain a functor
\[ \gamma\colon\STam\rightarrow\Tam. \]
In fact, $\gamma$ is the left adjoint functor of $\mathcal{U}$. (For the finite group case, see \cite{Tam}.)
\end{rem}

\begin{rem}\label{RemFinTam}
As in Remark \ref{RemFinMack}, similarly to the category of Mackey functors, when $G$ is finite and $(\mathcal{C},\mathcal{O}_{\bullet})$ is the natural Mackey system, then $\Tam$ is denoted by $\mathit{Tam}(G)$.

If $G=\{e\}$ is trivial, the evaluation at the trivial one-point $G$-set $\{e\}$ gives an equivalence of categories
\[ \mathit{Tam}(\{e\})\overset{\simeq}{\longrightarrow}\Ring\ \  ;\ \ T\mapsto T(\{e\}). \]
Similarly, $\mathit{STam}(\{e\})$ is equivalent to $\SRing$.
\end{rem}

\subsection{Diagram Lemmas}
In this section, we introduce some diagram lemmas concerning exponential diagrams, which will be needed later. Their proofs are straightforward, and we omit them. For the finite group case, see \cite{Tam}.

\begin{lem}\label{LemA}
Let $\eta\in\COM(X,Y)$ be any morphism, and let $\Amo$, $\Amt$ be two objects in $\COC/X$. Assume
\[
\xy
(-14,6)*+{X}="0";
(-14,-6)*+{Y}="2";
(-1,6)*+{A_i}="4";
(9,6)*+{}="5";
(18,4.7)*+{X\underset{Y}{\times}\Pi_{\eta}(A_i)}="6";
(16,2)*+{}="7";
(16,-6)*+{\Pi_{\eta}(A_i)}="8";
(0,0)*+{\mathit{exp}}="10";
{\ar_{\eta} "0";"2"};
{\ar_{p_i} "4";"0"};
{\ar_>>>>>{\lambda_i} "5";"4"};
{\ar^>>>>{\rho_i} "7";"8"};
{\ar^{\pi_i} "8";"2"};
\endxy
\]
is the exponential diagram for each $i=1,2$.
If we let $A$ be the fibered product of $A_1$ and $A_2$ over $X$
\[
\xy
(-20,7.25)*+{A=}="-2";
(-8,3.4)*+{}="-1";
(-9.5,5.8)*+{A_1\underset{X}{\times}A_2}="0";
(-2,7)*+{}="1";
(8,7)*+{A_2}="2";
(-8,-7)*+{A_1}="4";
(8,-7)*+{Y}="6";
(10,-8)*+{,}="10";
(0,0)*+{\square}="8";
{\ar^>>>>>{\varpi_2} "1";"2"};
{\ar_>>>>>{\varpi_1} "-1";"4"};
{\ar^{p_2} "2";"6"};
{\ar_{p_1} "4";"6"};
\endxy
\]
then
\[
\xy
(-10,8)*+{\Pi_{\eta}(A)}="0";
(10,8)*+{\Pi_{\eta}(A_2)}="2";
(-10,-8)*+{\Pi_{\eta}(A_1)}="4";
(10,-8)*+{Y}="6";
(0,0)*+{\square}="8";
{\ar^{\Pi_{\eta}(\varpi_2)} "0";"2"};
{\ar_{\Pi_{\eta}(\varpi_1)} "0";"4"};
{\ar^{\pi_2} "2";"6"};
{\ar_{\pi_1} "4";"6"};
\endxy
\]
is also a pull-back diagram. If we put $p=p_1\circ\varpi_1=p_2\circ\varpi_2$ and $\pi=\pi_1\circ\Pi_{\eta}(\varpi_1)=\pi_2\circ\Pi_{\eta}(\varpi_2)$, then we have an exponential diagram
\[
\xy
(-14,6)*+{X}="0";
(-14,-6)*+{Y}="2";
(-1,6)*+{A}="4";
(9,6)*+{}="5";
(18,4.7)*+{X\underset{Y}{\times}\Pi_{\eta}(A)}="6";
(16,2)*+{}="7";
(16,-6)*+{\Pi_{\eta}(A)}="8";
(22,-7)*+{.}="9";
(1,0)*+{\mathit{exp}}="10";
{\ar_{\eta} "0";"2"};
{\ar_{p} "4";"0"};
{\ar_>>>>>{\lambda} "5";"4"};
{\ar^>>>>{\rho} "7";"8"};
{\ar^{\pi} "8";"2"};
\endxy
\]
\end{lem}

\begin{lem}\label{LemB}
Let
\[
\xy
(-12,6)*+{X}="0";
(-12,-6)*+{Y}="2";
(0,6)*+{A}="4";
(12,6)*+{Z}="6";
(25,4.6)*+{\cong X\underset{Y}{\times}\Pi_{\eta}(A)}="7";
(12,-6)*+{\Pi_{\eta}(A)}="8";
(0,0)*+{\mathit{exp}}="10";
{\ar_{\eta} "0";"2"};
{\ar_<<<{p} "4";"0"};
{\ar_{\lambda} "6";"4"};
{\ar^{\rho} "6";"8"};
{\ar^{\pi} "8";"2"};
\endxy
\]
be an exponential diagram. If we pull it back by a morphism $\zeta\in \C(Y\ppr,Y)$, then the obtained diagram
\[
\xy
(-12,6)*+{X\ppr}="0";
(-12,-6)*+{Y\ppr}="2";
(0,6)*+{A\ppr}="4";
(12,6)*+{Z\ppr}="6";
(12,-6)*+{(\Pi_{\eta}(A))\ppr}="8";
(0,0)*+{\mathit{exp}}="10";
{\ar_{\eta\ppr} "0";"2"};
{\ar_<<<{p\ppr} "4";"0"};
{\ar_{\lambda\ppr} "6";"4"};
{\ar^{\rho\ppr} "6";"8"};
{\ar^{\pi\ppr} "8";"2"};
\endxy
\]
is also an exponential diagram. Here, $(-)\ppr$ is the abbreviation of $-\underset{Y}{\times}Y\ppr$.
\end{lem}

\begin{lem}\label{LemC}
Let $(\ref{DiagExp2})$ be an exponential diagram, and let
\[
\xy
(-6,6)*+{A}="0";
(6,6)*+{A\ppr}="2";
(-6,-6)*+{Z}="4";
(6,-6)*+{X\ppr}="6";
(0,0)*+{\square}="8";
{\ar_{\zeta\ppr} "2";"0"};
{\ar_{p} "0";"4"};
{\ar^{p\ppr} "2";"6"};
{\ar^{\zeta} "6";"4"};
\endxy
\]
be a pull-back diagram with $p\in\COC(A,Z)$.
If we let
\[
\xy
(-14,6)*+{X\ppr}="0";
(-14,-6)*+{Y\ppr}="2";
(-1,6)*+{A\ppr}="4";
(7,6)*+{}="5";
(18,4.7)*+{X\ppr\underset{Y\ppr}{\times}\Pi_{\eta\ppr}(A\ppr)}="6";
(16,2)*+{}="7";
(16,-6)*+{\Pi_{\eta\ppr}(A\ppr)}="8";
(1,0)*+{\mathit{exp}}="10";
{\ar_{\eta\ppr} "0";"2"};
{\ar_<<<{p\ppr} "4";"0"};
{\ar_<<<{\lambda\ppr} "5";"4"};
{\ar^>>>>{\rho\ppr} "7";"8"};
{\ar^{\pi\ppr} "8";"2"};
\endxy
\]
be the exponential diagram, then
\[
\xy
(-14,6)*+{X}="0";
(-14,-6)*+{Y}="2";
(-1,6)*+{A}="4";
(7,6)*+{}="5";
(18,4.7)*+{X\ppr\underset{Y\ppr}{\times}\Pi_{\eta\ppr}(A\ppr)}="6";
(16,2)*+{}="7";
(16,-6)*+{\Pi_{\eta\ppr}(A\ppr)}="8";
(1,0)*+{\mathit{exp}}="10";
{\ar_{\eta} "0";"2"};
{\ar_<<<<{\xi\circ p} "4";"0"};
{\ar_<<<{\zeta\ppr\circ\lambda\ppr} "5";"4"};
{\ar^>>>>{\rho\ppr} "7";"8"};
{\ar^{\upsilon\circ\pi\ppr} "8";"2"};
\endxy
\]
is also an exponential diagram.
\end{lem}

\subsection{Tambarization of a Mackey functor}
Fix a Tambara system $\TS$.
In this section, we associate a Tambara functor to any Mackey functor. In fact, we construct a functor $\mathcal{S}$ from $\SMack$ to $\STam$ in Theorem \ref{Thm1}.

As a consequence, we obtain the following functors.
\begin{eqnarray*}
\gamma\circ \mathcal{S}\colon\SMack&\rightarrow&\Tam\\
\Mack&\rightarrow&\STam\\
\Mack&\rightarrow&\Tam
\end{eqnarray*}

\begin{dfn}\label{DefSM}
Let $M$ be a semi-Mackey functor on $(\mathcal{C},\mathcal{O}_{\bullet})$.
For any $X\in\Ob(\C)$, define category $M\HG/X$ by the following.
\begin{enumerate}
\item[{\rm (a)}] An object in $M\HG/X$ is a pair $\Am$ of $p\in\COC(A,X)$ and $m_A\in M(A)$.
\item[{\rm (b)}] A morphism from $\Amo$ to $\Amt$ is a morphism $f\in\C(A_1,A_2)$, such that $p_2\circ f=p_1$ and $M^{\ast}(f)(m_{A_2})=m_{A_1}$.
\end{enumerate}
\end{dfn}

For any two objects $\Amo$ and $\Amt$ in $M\HG/X$, define their sum and product as follows.
\begin{enumerate}
\item[{\rm (i)}] $\Amo\amalg\Amt=(A_1\amalg A_2\overset{p_1\cup p_2}{\longrightarrow}X,m_{A_1}\amalg m_{A_2})$,
\item[{\rm (ii)}] $\Amo\times\Amt=(A_1\underset{X}{\times}A_2\overset{p}{\rightarrow}X,m_{A_1}\star m_{A_2})$.
\end{enumerate}
Here, $m_{A_1}\amalg m_{A_2}\in M(A_1\amalg A_2)$ is the pull-back of $(m_{A_1},m_{A_2})\in M(A_1)\times M(A_2)$ by the isomorphism
\[ (M^{\ast}(\iota_1),M^{\ast}(\iota_2))\colon M(A_1\amalg A_2)\overset{\cong}{\rightarrow}M(A_1)\times M(A_2), \]
where $\iota_1\colon A_1\hookrightarrow A_1\amalg A_2$ and $\iota_2\colon A_2\hookrightarrow A_1\amalg A_2$ are the inclusions.
$m_{A_1}\star m_{A_2}$ is the product of $M^{\ast}(\varpi_1)(m_{A_1})$ and $M^{\ast}(\varpi_2)(m_{A_2})$ in $M(A_1\underset{X}{\times}A_2)$, where $\varpi_1$ and $\varpi_2$ are the projections in the following pull-back diagram, and $p=p_1\circ\varpi_1=p_2\circ\varpi_2$.
\begin{eqnarray*}
\xy
(-8,3.4)*+{}="-1";
(-9.5,5.8)*+{A_1\underset{X}{\times}A_2}="0";
(-2,7)*+{}="1";
(8,7)*+{A_2}="2";
(-8,-7)*+{A_1}="4";
(8,-7)*+{X}="6";
(0,0)*+{\square}="8";
{\ar^>>>>>{\varpi_2} "1";"2"};
{\ar_>>>>>{\varpi_1} "-1";"4"};
{\ar^{p_2} "2";"6"};
{\ar_{p_1} "4";"6"};
\endxy
\qquad\  p=p_1\circ\varpi_1=p_2\circ\varpi_2\\
m_{A_1}\star m_{A_2}=M^{\ast}(\varpi_1)(m_{A_1})\cdot M^{\ast}(\varpi_2)(m_{A_2})
\end{eqnarray*}

The isomorphism classes of objects in $M\HG/X$ forms a semi-ring with these sums and products, which we denote by $\mathcal{S}_M(X)$.

\begin{prop}\label{PropForThm1}
$\mathcal{S}_M$ carries a natural structure of a semi-Tambara functor.
\end{prop}
\begin{proof}
First we construct structure morphisms for $\mathcal{S}_M$. Let $X,Y$ be any pair of objects in $\C$, and let $\Am\in \mathcal{S}_M(X)$ be any element.
\begin{enumerate}
\item For any $\zeta\in\C(Y,X)$, define $\zeta^{\ast}\colon \mathcal{S}_M(X)\rightarrow \mathcal{S}_M(Y)$
by
\[ \zeta^{\ast}\Am=(A\ppr\overset{p\ppr}{\rightarrow}Y,M^{\ast}(\zeta\ppr)(m_A)), \]
where
\begin{equation}
\xy
(-6,6)*+{A\ppr}="0";
(6,6)*+{Y}="2";
(-6,-6)*+{A}="4";
(6,-6)*+{X}="6";
(0,0)*+{\square}="8";
{\ar^{p\ppr} "0";"2"};
{\ar_{\zeta\ppr} "0";"4"};
{\ar^{\zeta} "2";"6"};
{\ar_{p} "4";"6"};
\endxy
\label{Diag+1}
\end{equation}
is a pull-back diagram.
\item For any $\xi\in\COC(X,Y)$, define $\xi_+\colon \mathcal{S}_M(X)\rightarrow \mathcal{S}_M(Y)$ by
\[ \xi_+\Am=(A\overset{\xi\circ p}{\longrightarrow}Y,m_A). \]
\item For any $\eta\in\COM(X,Y)$, define $\eta_{\bullet}\colon \mathcal{S}_M(X)\rightarrow \mathcal{S}_M(Y)$ by
\[ \eta_{\bullet}\Am=(Y\ppr\overset{\upsilon}{\longrightarrow}Y,M_{\ast}(\rho)M^{\ast}(\lambda)(m_A)), \]
where
\begin{equation}
\xy
(-12,6)*+{X}="0";
(-12,-6)*+{Y}="2";
(0,6)*+{A}="4";
(12,6)*+{X^{\prime}}="6";
(12,-6)*+{Y^{\prime}}="8";
(0,0)*+{\mathit{exp}}="10";
{\ar_{\eta} "0";"2"};
{\ar_{p} "4";"0"};
{\ar_{\lambda} "6";"4"};
{\ar^{\rho} "6";"8"};
{\ar^{\upsilon} "8";"2"};
\endxy
\label{Diag+2}
\end{equation}
is an exponential diagram.
\end{enumerate}

We only demonstrate the following. The other conditions can be shown easily.
\begin{enumerate}
\item[{\rm (i)}] $\xi_+$ is an additive homomorphism for any $\xi\in\COC(X,Y)$.
\item[{\rm (ii)}] $\eta_{\bullet}$ is a multiplicative homomorphism for any $\eta\in\COM(X,Y)$.
\item[{\rm (iii)}] $\zeta^{\ast}$ is a semi-ring homomorphism for any $\zeta\in\C(Y,X)$.
\item[{\rm (iv)}] $\mathcal{S}_M^{\ast}$ is additive, namely, for any $X,Y\in\Ob(\C)$,
\[ (\iota_X^{\ast},\iota_Y^{\ast})\colon \mathcal{S}_M(X\amalg Y)\rightarrow \mathcal{S}_M(X)\times \mathcal{S}_M(Y) \]
is an isomorphism, where $\iota_X$, $\iota_Y$ are the inclusions.
\item[{\rm (v)}] For any pull-back diagram in $\C$
\[
\xy
(-6,6)*+{X\ppr}="0";
(6,6)*+{X}="2";
(-6,-6)*+{Y\ppr}="4";
(6,-6)*+{Y}="6";
(0,0)*+{\square}="8";
{\ar^{\zeta\ppr} "0";"2"};
{\ar_{\xi\ppr} "0";"4"};
{\ar^{\xi} "2";"6"};
{\ar_{\zeta} "4";"6"};
\endxy
\]
with $\xi\in\COC(X,Y)$, we have $\zeta^{\ast}\circ\xi_+=\xi\ppr_+\circ\zeta^{\prime\ast}$.
\item[{\rm (vi)}] For any pull-back diagram in $\C$
\[
\xy
(-6,6)*+{X\ppr}="0";
(6,6)*+{X}="2";
(-6,-6)*+{Y\ppr}="4";
(6,-6)*+{Y}="6";
(0,0)*+{\square}="8";
{\ar^{\zeta\ppr} "0";"2"};
{\ar_{\eta\ppr} "0";"4"};
{\ar^{\eta} "2";"6"};
{\ar_{\zeta} "4";"6"};
\endxy
\]
with $\eta\in\COM(X,Y)$, we have $\zeta^{\ast}\circ\eta_{\bullet}=\eta\ppr_{\bullet}\circ\zeta^{\prime\ast}$.
\item[{\rm (vii)}] For any exponential diagram $(\ref{DiagExp2})$, we have
\[ \upsilon_+\circ\eta\ppr_{\bullet}\circ\zeta^{\ast}=\eta_{\bullet}\circ\xi_+. \]
\end{enumerate}

\smallskip

\noindent\underline{Confirmation of {\rm (i)}, {\rm (ii)}, {\rm (iii)}}

\smallskip

Let $\Amo$ and $\Amt$ be any pair of elements in $\mathcal{S}_M(X)$.
{\rm (i)} is trivially satisfied, since
\begin{eqnarray*}
\xi_+(\Amo\amalg\Amt)&=&(A_1\amalg A_2\overset{\xi\circ(p_1\cup p_2)}{\longrightarrow}X,m_{A_1}\amalg m_{A_2})\\
&=&\xi_+\Amo\amalg\xi_+\Amt.
\end{eqnarray*}
To show {\rm (iii)}, take the following pull-backs.
\begin{eqnarray*}
\xy
(-6,6)*+{A_i\ppr}="0";
(6,6)*+{A_i}="2";
(-6,-6)*+{Y}="4";
(6,-6)*+{X}="6";
(0,0)*+{\square}="8";
{\ar^{\zeta_i} "0";"2"};
{\ar_{p_i\ppr} "0";"4"};
{\ar^{p_i} "2";"6"};
{\ar_{\zeta} "4";"6"};
\endxy&
\xy
(-6,6)*+{A}="0";
(6,6)*+{A_2}="2";
(-6,-6)*+{A_1}="4";
(6,-6)*+{X}="6";
(0,0)*+{\square}="8";
{\ar^{\varpi_2} "0";"2"};
{\ar_{\varpi_1} "0";"4"};
{\ar^{p_2} "2";"6"};
{\ar_{p_1} "4";"6"};
\endxy&
\xy
(-6,6)*+{A\ppr}="0";
(6,6)*+{A}="2";
(-6,-6)*+{Y}="4";
(6,-6)*+{X}="6";
(0,0)*+{\square}="8";
{\ar^{\zeta\ppr} "0";"2"};
{\ar_{p\ppr} "0";"4"};
{\ar^{p} "2";"6"};
{\ar_{\zeta} "4";"6"};
\endxy
\\
(i=1,2)\ \ \ &
p=p_1\circ\varpi_1=p_2\circ\varpi_2&
\ \ 
\end{eqnarray*}
From these, we obtain the following pull-back diagrams.
\[
\xy
(-6,6)*+{A\ppr}="0";
(6,6)*+{A_i\ppr}="2";
(-6,-6)*+{A}="4";
(6,-6)*+{A_i}="6";
(0,0)*+{\square}="8";
{\ar^{\varpi_i\ppr} "0";"2"};
{\ar_{\zeta\ppr} "0";"4"};
{\ar^{\zeta_i} "2";"6"};
{\ar_{\varpi_i} "4";"6"};
\endxy
\ \ (i=1,2),
\ \ \ \ 
\xy
(-6,6)*+{A\ppr}="0";
(6,6)*+{A_2\ppr}="2";
(-6,-6)*+{A_1\ppr}="4";
(6,-6)*+{Y}="6";
(0,0)*+{\square}="8";
{\ar^{\varpi_2\ppr} "0";"2"};
{\ar_{\varpi_1\ppr} "0";"4"};
{\ar^{p_1\ppr} "2";"6"};
{\ar_{p_2\ppr} "4";"6"};
\endxy
\]

Thus we have
\begin{eqnarray*}
&&\hspace{-2cm}\zeta^{\ast}(\Amo\times\Amt)\\
&=&\zeta^{\ast}(A\overset{p}{\rightarrow}X,M^{\ast}(\varpi_1)(m_{A_1})\cdot M^{\ast}(\varpi_2)(m_{A_2}))\\
&=&(A\ppr\overset{p\ppr}{\rightarrow}Y,M^{\ast}(\zeta\ppr)M^{\ast}(\varpi_1)(m_{A_1})\cdot M^{\ast}(\zeta\ppr)M^{\ast}(\varpi_2)(m_{A_2}))\\
&=&(A\ppr\overset{p\ppr}{\rightarrow}Y,M^{\ast}(\varpi\ppr_1)M^{\ast}(\zeta_1)(m_{A_1})\cdot M^{\ast}(\varpi\ppr_2)M^{\ast}(\zeta_2)(m_{A_2})\\
&=&\zeta^{\ast}\Amo\times\zeta^{\ast}\Amt.
\end{eqnarray*}
Moreover, since
\[
\xy
(-10,6)*+{A_1\ppr\amalg A_2\ppr}="0";
(10,6)*+{Y}="2";
(-10,-6)*+{A_1\amalg A_2}="4";
(10,-6)*+{X}="6";
(0,0)*+{\square}="8";
{\ar^>>>>>>{p_1\ppr\cup p_2\ppr} "0";"2"};
{\ar_{\zeta_1\amalg\zeta_2} "0";"4"};
{\ar^{\zeta} "2";"6"};
{\ar_>>>>>>{p_1\cup p_2} "4";"6"};
\endxy
\]
is a pull-back diagram, we have
\begin{eqnarray*}
&&\hspace{-2cm}\zeta^{\ast}(\Amo\amalg\Amt)\\
&=&(A\ppr_1\amalg A\ppr_2\overset{p\ppr_1\cup p\ppr_2}{\longrightarrow}Y,M^{\ast}(\zeta_1\amalg\zeta_2)(m_{A_1}\amalg m_{A_2}))\\
&=&(A\ppr_1\amalg A\ppr_2\overset{p\ppr_1\cup p\ppr_2}{\longrightarrow}Y,M^{\ast}(\zeta_1)(m_{A_1})\amalg M^{\ast}(\zeta_2)(m_{A_2}))\\
&=&\zeta^{\ast}\Amo\amalg\zeta^{\ast}\Amt.
\end{eqnarray*}

To show {\rm (ii)}, we use the notation in Lemma \ref{LemA}.
For any pair of elements $\Amo$ and $\Amt$ in $S_M(X)$, we have
\begin{eqnarray*}
&&\hspace{-2cm}\eta_{\bullet}(\Amo\cdot\Amt)\\
&=&\eta_{\bullet}(A\overset{p}{\rightarrow}X,M^{\ast}(\varpi_1)(m_{A_1})\cdot M^{\ast}(\varpi_2)(m_{A_2}))\\
&=&(\Pi_{\eta}(A)\overset{\pi}{\rightarrow}Y,M_{\ast}(\rho)M^{\ast}(\lambda)(M^{\ast}(\varpi_1)(m_{A_1})\cdot M^{\ast}(\varpi_2)(m_{A_2}))),
\end{eqnarray*}
\begin{eqnarray*}
&&\hspace{-1cm}\eta_{\bullet}\Amo\cdot\eta_{\bullet}\Amt\\
&=&(\Pi_{\eta}(A_1)\overset{\pi_1}{\rightarrow}Y,M_{\ast}(\rho_1)M^{\ast}(\lambda_1)(m_{A_1}))\cdot(\Pi_{\eta}(A_2)\overset{\pi_2}{\rightarrow}Y,M_{\ast}(\rho_2)M^{\ast}(\lambda_2)(m_{A_2})).
\end{eqnarray*}
Since we have
\[ M_{\ast}(\rho)M^{\ast}(\lambda)M^{\ast}(\varpi_i)=M^{\ast}(\Pi_{\eta}(\varpi_i))M_{\ast}(\rho_i)M^{\ast}(\lambda_i) \]
\[
\xy
(-28,7)*+{A_1\times_XA_2}="1";
(0,7)*+{X\times_Y\Pi_{\eta}(A)}="2";
(28,7)*+{\Pi_{\eta}(A)}="3";
(-28,-7)*+{A_i}="11";
(0,-7)*+{X\times_Y\Pi_{\eta}(A_i)}="12";
(28,-7)*+{\Pi_{\eta}(A_i)}="13";
(14,0)*+{\square}="10";
{\ar_{\lambda} "2";"1"};
{\ar^{\rho} "2";"3"};
{\ar_{\varpi_i} "1";"11"};
{\ar|*+{_{1_X\times_Y\Pi_{\eta}(\varpi_i)}} "2";"12"};
{\ar^{\Pi_{\eta}(\varpi_i)} "3";"13"};
{\ar^{\lambda_i} "12";"11"};
{\ar_{\rho_i} "12";"13"};
{\ar@{}|\circlearrowright "1";"12"};
\endxy
\]
for $i=1,2$, we obtain
\[ \eta_{\bullet}(\Amo\cdot\Amt)=\eta_{\bullet}\Amo\cdot\eta_{\bullet}\Amt. \]

\smallskip

\noindent\underline{Confirmation of {\rm (iv)}}

\smallskip

To each $(A\overset{p}{\rightarrow}X,m_A)\in \mathcal{S}_M(X)$ and $(B\overset{q}{\rightarrow}Y,m_B)\in \mathcal{S}_M(Y)$, associate
\[ (A\amalg B\overset{p\amalg q}{\longrightarrow}X\amalg Y,M_{\ast}(\iota_A)(m_A)\cdot M_{\ast}(\iota_B)(m_B))\in\mathcal{S}_M(X\amalg Y), \]
where $\iota_A\colon A\hookrightarrow A\amalg B$ and $\iota_B\colon B\hookrightarrow A\amalg B$ are the inclusions. This gives the inverse of $(\iota_X^{\ast},\iota_Y^{\ast})$.

\smallskip

\noindent\underline{Confirmation of {\rm (v)}}

\smallskip

Let $\Am$ be any element in $\mathcal{S}_M(X)$. If we take pull-backs
\[
\xy
(-12,6)*+{A\ppr}="1";
(0,6)*+{X\ppr}="2";
(12,6)*+{Y\ppr}="3";
(-12,-6)*+{A}="11";
(0,-6)*+{X}="12";
(12,-6)*+{Y}="13";
(14,-7)*+{,}="14";
(-6,0)*+{\square}="0";
(6,0)*+{\square}="10";
{\ar^{p\ppr} "1";"2"};
{\ar^{\xi\ppr} "2";"3"};
{\ar_{\zeta^{\prime\prime}} "1";"11"};
{\ar_{\zeta\ppr} "2";"12"};
{\ar^{\zeta} "3";"13"};
{\ar_{p} "11";"12"};
{\ar_{\xi} "12";"13"};
\endxy
\]
then we have
\begin{eqnarray*}
\zeta^{\ast}\xi_+\Am&=&(A\ppr\overset{\xi\ppr\circ p\ppr}{\longrightarrow}Y\ppr,M^{\ast}(\zeta^{\prime\prime})(m_A))\\
&=&\xi\ppr_+\zeta^{\prime\ast}\Am.
\end{eqnarray*}

\smallskip

\noindent\underline{Confirmation of {\rm (vi)}}

\smallskip

We use the notation in Lemma \ref{LemB}. Let $\zeta^{\prime\prime}\in\C(A\ppr,A)$ be the pull-back of $\zeta\ppr$ (or $\zeta$).
\[
\xy
(-12,6)*+{A\ppr}="1";
(0,6)*+{X\ppr}="2";
(12,6)*+{Y\ppr}="3";
(-12,-6)*+{A}="11";
(0,-6)*+{X}="12";
(12,-6)*+{Y}="13";
(-6,0)*+{\square}="0";
(6,0)*+{\square}="10";
{\ar^{p\ppr} "1";"2"};
{\ar^{\eta\ppr} "2";"3"};
{\ar_{\zeta^{\prime\prime}} "1";"11"};
{\ar_{\zeta\ppr} "2";"12"};
{\ar^{\zeta} "3";"13"};
{\ar_{p} "11";"12"};
{\ar_{\eta} "12";"13"};
\endxy
\]
If we let $\mathrm{pr}_Z$ and $\mathrm{pr}_{\Pi}$ be the canonical projections induced from $\zeta$
\[
\xy
(-16,7)*+{Z\ppr}="1";
(0,7)*+{(\Pi_{\eta}(A))\ppr}="2";
(16,7)*+{Y\ppr}="3";
(-16,-7)*+{Z}="11";
(0,-7)*+{\Pi_{\eta}(A)}="12";
(16,-7)*+{Y}="13";
(-8,0)*+{\square}="0";
(8,0)*+{\square}="10";
{\ar^>>>{\rho\ppr} "1";"2"};
{\ar^>>>{\pi\ppr} "2";"3"};
{\ar_{\mathrm{pr}_Z} "1";"11"};
{\ar_{\mathrm{pr}_{\Pi}} "2";"12"};
{\ar^{\zeta} "3";"13"};
{\ar_<<<<<{\rho} "11";"12"};
{\ar_<<<<<{\eta} "12";"13"};
\endxy
\]
then
\[
\xy
(-6,6)*+{Z\ppr}="0";
(6,6)*+{A\ppr}="2";
(-6,-6)*+{Z}="4";
(6,-6)*+{A}="6";
{\ar^{\lambda\ppr} "0";"2"};
{\ar_{\mathrm{pr}_Z} "0";"4"};
{\ar^{\zeta^{\prime\prime}} "2";"6"};
{\ar_{\lambda} "4";"6"};
{\ar@{}|\circlearrowright "0";"6"};
\endxy
\]
is commutative.
By Lemma \ref{LemB}, we have
\begin{eqnarray*}
\eta\ppr_{\bullet}\zeta^{\prime\ast}\Am&=&((\Pi_{\eta}(A))\ppr\overset{\pi\ppr}{\rightarrow}Y\ppr,M_{\ast}(\rho\ppr)M^{\ast}(\lambda^{\prime})M^{\ast}(\zeta^{\prime\prime})(m_A)),\\
\zeta^{\ast}\eta_{\bullet}\Am&=&((\Pi_{\eta}(A))\ppr\overset{\pi\ppr}{\rightarrow}Y\ppr,M^{\ast}(\mathrm{pr}_{\Pi})M_{\ast}(\rho)M^{\ast}(\lambda)(m_A)).
\end{eqnarray*}
Since
\begin{eqnarray*}
M_{\ast}(\rho\ppr)M^{\ast}(\lambda\ppr)M^{\ast}(\zeta^{\prime\prime})&=&M_{\ast}(\rho\ppr)M^{\ast}(\mathrm{pr}_Z)M^{\ast}(\lambda)\\
&=&M^{\ast}(\mathrm{pr}_{\Pi})M_{\ast}(\rho)M^{\ast}(\lambda),
\end{eqnarray*}
we obtain $\eta\ppr_{\bullet}\circ\zeta^{\prime\ast}=\zeta^{\ast}\circ\eta_{\bullet}$.

\smallskip

\noindent\underline{Confirmation of {\rm (vii)}}

\smallskip

In the notation of Lemma \ref{LemC}, for any $(A\overset{p}{\rightarrow}Z,m_A)\in \mathcal{S}_M(Z)$, we have
\begin{eqnarray*}
\eta_{\bullet}\xi_+(A\overset{p}{\rightarrow}Z,m_A)&=&\eta_{\bullet}(A\overset{\xi\circ p}{\longrightarrow}X,m_A)\\
&=&(\Pi_{\eta\ppr}(A\ppr)\overset{\upsilon\circ\pi\ppr}{\longrightarrow}Y,M_{\ast}(\rho\ppr)M^{\ast}(\zeta\ppr\circ\lambda\ppr)(m_A))\\
&=&\upsilon_+(\Pi_{\eta\ppr}(A\ppr)\overset{\pi\ppr}{\longrightarrow}Y\ppr,M_{\ast}(\rho\ppr)M^{\ast}(\lambda\ppr)M^{\ast}(\zeta\ppr)(m_A))\\
&=&\upsilon_+\eta\ppr_{\bullet}(A\ppr\overset{p\ppr}{\rightarrow}X\ppr,M^{\ast}(\zeta\ppr)(m_A))\\
&=&\upsilon_+\eta\ppr_{\bullet}\zeta^{\ast}(A\overset{p}{\rightarrow}Z,m_A).
\end{eqnarray*}
\end{proof}

\begin{thm}\label{Thm1}
Let $\TS$ be a Tambara system on $G$. The construction of $\mathcal{S}_M$ in Proposition \ref{PropForThm1} gives a functor
\[ \mathcal{S}\colon\SMack\rightarrow\STam. \]
\end{thm}
\begin{proof}
By Proposition \ref{PropForThm1}, $\mathcal{S}_M$ becomes an object in $\STam$ for any $M\in\Ob(\SMack)$.

It suffices to construct a morphism $\mathcal{S}_{\varphi}\in\STam(\mathcal{S}_M,\mathcal{S}_N)$ for each $\varphi\in\SMack(M,N)$ where $M,N\in\Ob(\SMack)$, in a functorial way.
For each $X\in\Ob(\C)$, we have a functor from $M\HG/X$ to $N\HG/X$ defined by
\begin{eqnarray*}
M\HG/X&\rightarrow&N\HG/X\\
\Am&\mapsto&(A\overset{p}{\rightarrow}X,\varphi_A(m_A)),\\
\end{eqnarray*}
where $\varphi_A$ is the component of $\varphi$ at $A$ (monoid homomorphism from $M(A)$ to $N(A)$).
This functor preserves sums and products, and thus we obtain a semi-ring homomorphism
\[ \mathcal{S}_{\varphi}(X)\colon \mathcal{S}_M(X)\rightarrow \mathcal{S}_N(X). \]
\begin{claim}\label{ClaimForThm1}
$\{ \mathcal{S}_{\varphi}(X)\mid X\in\Ob(\C) \}$ is compatible with structure morphisms of $\mathcal{S}_M$ and $\mathcal{S}_N$.
\end{claim}
If this claim is shown, it means $\mathcal{S}_{\varphi}\in\STam(\mathcal{S}_M,\mathcal{S}_N)$. Moreover, it can be easily checked that the correspondence $\varphi\mapsto \mathcal{S}_{\varphi}$ preserves the identities and compositions, so we obtain a functor $\mathcal{S}\colon\SMack\rightarrow\STam$. Thus it suffices to show Claim \ref{ClaimForThm1}.

\begin{proof}[Proof of Claim \ref{ClaimForThm1}]
Let $\Am$ be any element in $\mathcal{S}_M(X)$.

\smallskip

\noindent\underline{Compatibility with $\xi_+$}

\smallskip

Let $\xi\in\COC(X,Y)$ be any morphism.
Obviously we have
\begin{eqnarray*}
\mathcal{S}_{\varphi}(Y)\circ\xi_+\Am&=&(A\overset{\xi\circ p}{\longrightarrow}Y,\varphi_A(m_A))\\
&=&\xi_+\circ \mathcal{S}_{\varphi}(X)\Am.
\end{eqnarray*}
\[
\xy
(-12,6)*+{\mathcal{S}_M(X)}="0";
(12,6)*+{\mathcal{S}_N(X)}="2";
(-12,-6)*+{\mathcal{S}_M(Y)}="4";
(12,-6)*+{\mathcal{S}_N(Y)}="6";
{\ar^{\mathcal{S}_{\varphi}(X)} "0";"2"};
{\ar_{\xi_+} "0";"4"};
{\ar^{\xi_+} "2";"6"};
{\ar_{\mathcal{S}_{\varphi}(Y)} "4";"6"};
{\ar@{}|\circlearrowright "0";"6"};
\endxy
\]

\smallskip

\noindent\underline{Compatibility with $\eta_{\bullet}$}

\smallskip

Let $\eta\in\COM(X,Y)$ be any morphism, and let $(\ref{DiagExp1})$ be the exponential diagram.
Then we have
\begin{eqnarray*}
\mathcal{S}_{\varphi}(Y)\circ\eta_{\bullet}\Am&=&(\Pi_{\eta}(A)\overset{\pi}{\rightarrow}Y,\varphi_{\Pi_{\eta}(A)}M_{\ast}(\rho)M^{\ast}(\lambda)(m_A))\\
&=&(\Pi_{\eta}(A)\overset{\pi}{\rightarrow}Y,N_{\ast}(\rho)N^{\ast}(\lambda)\varphi_A(m_A))\\
&=&\eta_{\bullet}\circ \mathcal{S}_{\varphi}(X)\Am.
\end{eqnarray*}
\[
\xy
(-12,6)*+{\mathcal{S}_M(X)}="0";
(12,6)*+{\mathcal{S}_N(X)}="2";
(-12,-6)*+{\mathcal{S}_M(Y)}="4";
(12,-6)*+{\mathcal{S}_N(Y)}="6";
{\ar^{\mathcal{S}_{\varphi}(X)} "0";"2"};
{\ar_{\eta_{\bullet}} "0";"4"};
{\ar^{\eta_{\bullet}} "2";"6"};
{\ar_{\mathcal{S}_{\varphi}(Y)} "4";"6"};
{\ar@{}|\circlearrowright "0";"6"};
\endxy
\]

\smallskip

\noindent\underline{Compatibility with $\zeta^{\ast}$}

\smallskip

Let $\zeta\in\C(Y,X)$ be any morphism, and let $(\ref{Diag+1})$
be the pull-back diagram. Then we have
\begin{eqnarray*}
\mathcal{S}_{\varphi}(Y)\circ\zeta^{\ast}\Am&=&(A\ppr\overset{p\ppr}{\rightarrow}Y,\varphi_{A\ppr}M^{\ast}(\zeta\ppr)(m_A))\\
&=&(A\ppr\overset{p\ppr}{\rightarrow}Y,N^{\ast}(\zeta\ppr)\varphi_{A}(m_A))\\
&=&\zeta^{\ast}\mathcal{S}_{\varphi}(X)\Am.
\end{eqnarray*}
\[
\xy
(-12,6)*+{\mathcal{S}_M(X)}="0";
(12,6)*+{\mathcal{S}_N(X)}="2";
(-12,-6)*+{\mathcal{S}_M(Y)}="4";
(12,-6)*+{\mathcal{S}_N(Y)}="6";
{\ar^{\mathcal{S}_{\varphi}(X)} "0";"2"};
{\ar_{\zeta^{\ast}} "0";"4"};
{\ar^{\zeta^{\ast}} "2";"6"};
{\ar_{\mathcal{S}_{\varphi}(Y)} "4";"6"};
{\ar@{}|\circlearrowright "0";"6"};
\endxy
\]
\end{proof}
\end{proof}

\begin{cor}
Composing $\mathcal{S}$ with $\gamma\colon\STam\rightarrow\Tam$, we obtain a functor
\[ \mathcal{T}=\gamma\circ\mathcal{S}\colon\SMack\rightarrow\Tam. \]
Moreover, since $\Mack$ is a full subcategory of $\SMack$, we also obtain functors
\begin{eqnarray*}
\Mack&\rightarrow&\STam,\\
\Mack&\rightarrow&\Tam.
\end{eqnarray*}
\end{cor}

\subsection{Adjoint property}


\begin{thm}\label{Thm2}
Let $\TS$ be any Tambara system on $G$.
The functor constructed in Theorem \ref{Thm1}
\[ \mathcal{S}\colon\SMack\rightarrow\STam \]
is left adjoint to the forgetful functor $\mu\colon\STam\rightarrow\SMack$.
\end{thm}
\begin{proof}
We construct natural maps
\begin{eqnarray*}
\Phi\colon\STam(\mathcal{S}_M,T)&\rightarrow&\SMack(M,T^{\mu})\\
\Psi\colon\SMack(M,T^{\mu})&\rightarrow&\STam(\mathcal{S}_M,T)
\end{eqnarray*}
for each $M\in\Ob(\SMack)$, $T\in\Ob(\STam)$, and show $\Psi$ and $\Phi$ are mutually inverse.

Let $\psi\in\STam(\mathcal{S}_M,T)$ be any morphism. For each $X\in\Ob(\C)$, define $\Phi(\psi)_X\colon M(X)\rightarrow T^{\mu}(X)$ by
\[
\Phi(\psi)_X(m)=\psi_X(X\overset{\mathrm{id}_X}{\rightarrow}X,m)\qquad({}^{\forall}m\in M(X)).
\]

On the contrary, for each $\varphi\in\SMack(M,T^{\mu})$ and $X\in\Ob(\C)$, define $\Psi(\varphi)_X\colon \mathcal{S}_M(X)\rightarrow T(X)$ by
\begin{eqnarray*}
&\Psi(\varphi)_X\Am=T_+(p)\circ\varphi_A(m_A)&\\
&({}^{\forall}\Am\in \mathcal{S}_M(X)).&
\end{eqnarray*}

\begin{claim}\label{ClaimForThm2}
For any $\varphi\in\SMack(M,T^{\mu})$ and $\psi\in\STam(\mathcal{S}_M,T)$, the following are satisfied.
\begin{enumerate}
\item[{\rm (i)}] $\Phi(\psi):=\{ \Phi(\psi)_X\}_{X\in\Ob(\C)}$ belongs to $\SMack(M,T^{\mu})$.
\item[{\rm (ii)}] $\Psi(\varphi):=\{ \Psi(\varphi)_X\}_{X\in\Ob(\C)}$ belongs to $\STam(\mathcal{S}_M,T)$.
\item[{\rm (iii)}] $\Psi\circ\Phi(\psi)=\psi$, $\Phi\circ\Psi(\varphi)=\varphi$.
\end{enumerate}
\end{claim}

\begin{proof}[Proof of Claim \ref{ClaimForThm2}]
\noindent{{\rm (i)}} It suffices to show the commutativity of
\[
\mathrm{(ia)}\ 
\xy
(-9,6)*+{M(X)}="0";
(9,6)*+{T(X)}="2";
(-9,-6)*+{M(Y)}="4";
(9,-6)*+{T(Y)}="6";
(14,-7)*+{,}="10";
{\ar^{\Phi(\psi)_X} "0";"2"};
{\ar_{M_{\ast}(\eta)} "0";"4"};
{\ar^{T_{\bullet}(\eta)} "2";"6"};
{\ar_{\Phi(\psi)_Y} "4";"6"};
\endxy
\qquad\ \mathrm{(ib)}\ 
\xy
(-9,6)*+{M(X)}="0";
(9,6)*+{T(X)}="2";
(-9,-6)*+{M(Y)}="4";
(9,-6)*+{T(Y)}="6";
(14,-7)*+{.}="10";
{\ar^{\Phi(\psi)_X} "0";"2"};
{\ar_{M^{\ast}(\zeta)} "0";"4"};
{\ar^{T^{\ast}(\zeta)} "2";"6"};
{\ar_{\Phi(\psi)_Y} "4";"6"};
\endxy
\]
for arbitrary morphisms $\eta\in\COM(X,Y)$ and $\zeta\in\C(Y,X)$.
Let $m\in M(X)$ be any element.

Commutativity of {\rm (ia)} follows from
\begin{eqnarray*}
T_{\bullet}(\eta)\Phi(\psi)_X(m)&=&T_{\bullet}(\eta)\psi_X(X\overset{\mathrm{id}_X}{\rightarrow}X,m)\\
&=&\psi_Y\eta_{\bullet}(X\overset{\mathrm{id}_X}{\rightarrow}X,m)\\
&=&\psi_Y(Y\overset{\mathrm{id}_Y}{\rightarrow}Y,M_{\ast}(\eta)(m))\\
&=&\Phi(\psi)_YM_{\ast}(\eta)(m),
\end{eqnarray*}
since
\[
\xy
(-12,6)*+{X}="0";
(-12,-6)*+{Y}="2";
(0,6)*+{X}="4";
(12,6)*+{X}="6";
(12,-6)*+{Y}="8";
(0,0)*+{\mathit{exp}}="10";
{\ar_{\eta} "0";"2"};
{\ar_{\mathrm{id}_X} "4";"0"};
{\ar_{\mathrm{id}_X} "6";"4"};
{\ar^{\eta} "6";"8"};
{\ar^{\mathrm{id}_Y} "8";"2"};
\endxy
\]
is an exponential diagram.

Commutativity of {\rm (ib)} follows from
\begin{eqnarray*}
T^{\ast}(\zeta)\Phi(\psi)_X(m)&=&T^{\ast}(\zeta)\psi_X(X\overset{\mathrm{id}_X}{\rightarrow}X,m)\\
&=&\psi_Y\zeta^{\ast}(X\overset{\mathrm{id}_X}{\rightarrow}X,m)\\
&=&\psi_Y(Y\overset{\mathrm{id}_Y}{\rightarrow}Y,M^{\ast}(\zeta)(m))\\
&=&\Phi(\psi)_YM^{\ast}(\zeta)(m).
\end{eqnarray*}

\noindent {\rm (ii)}
It suffices to show the commutativity of
\begin{eqnarray*}
&\mathrm{(iia)}
\xy
(-10,6)*+{\mathcal{S}_M(X)}="0";
(10,6)*+{T(X)}="2";
(-10,-6)*+{\mathcal{S}_M(Y)}="4";
(10,-6)*+{T(Y)}="6";
(15,-7)*+{,}="10";
{\ar^{\Psi(\varphi)_X} "0";"2"};
{\ar_{\xi_+} "0";"4"};
{\ar^{T_+(\xi)} "2";"6"};
{\ar_{\Psi(\varphi)_Y} "4";"6"};
\endxy
\qquad\ \mathrm{(iib)}
\xy
(-10,6)*+{\mathcal{S}_M(X)}="0";
(10,6)*+{T(X)}="2";
(-10,-6)*+{\mathcal{S}_M(Y)}="4";
(10,-6)*+{T(Y)}="6";
(15,-7)*+{,}="10";
{\ar^{\Psi(\varphi)_X} "0";"2"};
{\ar_{\eta_{\bullet}} "0";"4"};
{\ar^{T_{\bullet}(\eta)} "2";"6"};
{\ar_{\Psi(\varphi)_Y} "4";"6"};
\endxy
&
\\
&\mathrm{(iic)}
\xy
(-10,6)*+{\mathcal{S}_M(X)}="0";
(10,6)*+{T(X)}="2";
(-10,-6)*+{\mathcal{S}_M(Y)}="4";
(10,-6)*+{T(Y)}="6";
(15,-7)*+{.}="10";
{\ar^{\Psi(\varphi)_X} "0";"2"};
{\ar_{\zeta^{\ast}} "0";"4"};
{\ar^{T^{\ast}(\zeta)} "2";"6"};
{\ar_{\Psi(\varphi)_Y} "4";"6"};
\endxy
&
\end{eqnarray*}
for arbitrary $\xi\in\COC(X,Y)$, $\eta\in\COM(X,Y)$ and $\zeta\in\C(Y,X)$.

Let $\Am$ be any element in $\mathcal{S}_M(X)$. The commutativity of {\rm (iia)} follows from
\begin{eqnarray*}
T_+(\xi)\Psi(\varphi)_X\Am&=&T_+(\xi)T_+(p)\varphi_A(m_A)\\
&=&\Psi(\varphi)_Y(A\overset{\xi\circ p}{\longrightarrow}Y,m_A)\\
&=&\Psi(\varphi)_Y\xi_+\Am.
\end{eqnarray*}

In the notation of exponential diagram $(\ref{Diag+2})$, the commutativity of {\rm (iib)} follows from
\begin{eqnarray*}
T_{\bullet}(\eta)\Psi(\varphi)_X\Am&=&T_{\bullet}(\eta)T_+(p)\varphi_A(m_A)\\
&=&T_+(\upsilon)T_{\bullet}(\rho)T^{\ast}(\lambda)\varphi_A(m_A)\\
&=&T_+(\upsilon)\varphi_{Y\ppr}M_{\ast}(\rho)M^{\ast}(\lambda)(m_A)\\
&=&\Psi(\varphi)_Y(Y\ppr\overset{\upsilon}{\rightarrow}Y,M_{\ast}(\rho)M^{\ast}(\lambda)(m_A))\\
&=&\Psi(\varphi)_Y\eta_{\bullet}\Am.
\end{eqnarray*}
In the notation of $(\ref{Diag+1})$, the commutativity of {\rm (iic)} follows from
\begin{eqnarray*}
T^{\ast}(\zeta)\Psi(\varphi)_X\Am&=&T^{\ast}(\zeta)T_+(p)\varphi_A(m_A)\\
&=&T_+(p\ppr)T^{\ast}(\zeta\ppr)\varphi_A(m_A)\\
&=&T_+(p\ppr)\varphi_{Y\ppr}M^{\ast}(\zeta\ppr)(m_A)\\
&=&\Psi(\varphi)_Y(A\ppr\overset{p\ppr}{\rightarrow}Y,M^{\ast}(\zeta\ppr)(m_A))\\
&=&\Psi(\varphi)_Y\zeta^{\ast}\Am.
\end{eqnarray*}

\noindent{\rm (iii)}
Let $X$ be any object in $\C$. For any $m\in M(X)$, we have
\begin{eqnarray*}
(\Phi\circ\Psi(\varphi))_X(m)&=&\Psi(\varphi)_X(X\overset{\mathrm{id}_X}{\rightarrow}X,m)\\
&=&T_+(\mathrm{id}_X)\varphi_X(m)=\varphi_X(m).
\end{eqnarray*}

Conversely, for any $\Am\in \mathcal{S}_M(X)$, we have
\begin{eqnarray*}
(\Psi\circ\Phi(\psi))_X\Am&=&T_+(p)\Phi(\psi)_A(m_A)\\
&=&T_+(p)\psi_A(A\overset{\mathrm{id}_A}{\rightarrow}A,m_A)\\
&=&\psi_Xp_+(A\overset{\mathrm{id}_A}{\rightarrow}A,m_A)\\
&=&\psi_X\Am.
\end{eqnarray*}
\end{proof}
\end{proof}

\begin{cor}\label{CorOfThm2}Composition
\[ \mathcal{T}=\gamma\circ\mathcal{S}\colon\SMack\rightarrow\Tam \]
is left adjoint to the forgetful functor $\mu\circ \mathcal{U}\colon\Tam\rightarrow\SMack$. 
\end{cor}
\begin{proof}
This immediately follows from Remark \ref{RemForget} and Theorem \ref{Thm2}.
\end{proof}
With this corollary, $\mathcal{T}$ can be regarded as a $G$-bivariant analog of the functor taking monoid-rings. Indeed if $G$ is trivial, this is equivalent to the monoid-ring functor:
\[
\xy
(-14,6)*+{\mathit{SMack}(\{e\})}="0";
(14,6)*+{\Mon}="2";
(-14,-6)*+{\mathit{Tam}(\{e\})}="4";
(14,-6)*+{\Ring}="6";
{\ar^{\simeq} "0";"2"};
{\ar_{\mathcal{T}} "0";"4"};
{\ar^{\mathbb{Z}[-]} "2";"6"};
{\ar_{\simeq} "4";"6"};
{\ar@{}|\circlearrowright "0";"6"};
\endxy
\]

\section{Relation with other constructions}

In the rest, $G$ is assumed to be finite, and we only consider the natural Mackey system on $G$, as in Remark \ref{RemFinMack} and Remark \ref{RemFinTam}. In this case, we have $\GS=\C$. 
As in Remark \ref{FinInf}, we may work over the category $\sG$ of finite $G$-sets. For example, $M\HG/X$ in Definition \ref{DefSM} is replaced by $M\text{-}\sG/X$.

\subsection{Relation with the crossed Burnside ring}
\label{3-1}

\smallskip

Let $Q$ be a (not necessarily finite) $G$-monoid, and $\mathcal{P}_Q$ be the fixed point functor associated to $Q$ (Example \ref{ExFix}).
By Theorem \ref{Thm1}, we obtain a Tambara functor $\mathcal{T}_{\mathcal{P}_Q}$. Recall that we have a sequence of functors
\[ G\text{-}\Mon\overset{\mathcal{P}}{\longrightarrow}\SMackG\overset{\mathcal{T}}{\longrightarrow}\TamG, \]
where $G\text{-}\Mon$ is the category of $G$-monoids.

We show $\mathcal{T}_{\mathcal{P}_Q}$ generalizes the crossed Burnside ring functor in \cite{O-Y2}, \cite{O-Y3}. Indeed, if $Q$ is finite, we construct an isomorphism of Tambara functors between $\mathcal{T}_{\mathcal{P}_Q}$ and the crossed Burnside ring functor, in Proposition \ref{PropCBR}.

First we recall the definition of the crossed Burnside ring. 
\begin{dfn}\label{DefCrossBurn}
(\S 4.3 in \cite{O-Y2}) Fix a finite $G$-monoid $Q$. The category of crossed $G$-sets $\CrossG$ is defined as follows.
\begin{enumerate}
\item[{\rm (a)}] An object in $\CrossG$ is a triplet $\Am$ of a finite $G$-set $A$, $G$-maps $f\in\sG(A,X)$ and $m_A\in\sG(A,Q)$.
\item[{\rm (b)}] A morphism from $\Amo$ to $\Amt$ in $\CrossG$ is a $G$-map $f\in\sG(A_1,A_2)$ satisfying $p_2\circ f=p_1$ and $m_{A_2}\circ f=m_{A_1}$.
\[
\xy
(-8,6)*+{A_1}="0";
(8,6)*+{A_2}="2";
(0,-6)*+{X}="4";
(0,9)*+{}="6";
{\ar^{f} "0";"2"};
{\ar_{p_1} "0";"4"};
{\ar^{p_2} "2";"4"};
{\ar@{}|\circlearrowright "4";"6"};
\endxy
\quad
\xy
(-8,6)*+{A_1}="0";
(8,6)*+{A_2}="2";
(0,-6)*+{Q}="4";
(0,9)*+{}="6";
{\ar^{f} "0";"2"};
{\ar_{m_{A_1}} "0";"4"};
{\ar^{m_{A_2}} "2";"4"};
{\ar@{}|\circlearrowright "4";"6"};
\endxy
\]
\end{enumerate}
In $\CrossG$, for any pair of objects $\Ami$ $(i=1,2)$, their sum $(=${\it coproduct} in \cite{O-Y2}$)$ and product $(=${\it tensor product} in \cite{O-Y2}$)$ are
\begin{eqnarray*}
\Amo+\Amt&=&(A_1\amalg A_2\overset{p_1\cup p_2}{\longrightarrow}X,m_{A_1}\cup m_{A_2}),\\
\Amo\cdot\Amt&=&(A_1\times_X A_2\overset{p}{\rightarrow}X,m_{A_1}\ast m_{A_2}),
\end{eqnarray*}
where
\[
\xy
(-8,3.4)*+{}="-1";
(-9.5,5.8)*+{A_1\underset{X}{\times}A_2}="0";
(-2,7)*+{}="1";
(8,7)*+{A_2}="2";
(-8,-7)*+{A_1}="4";
(8,-7)*+{Y}="6";
(0,0)*+{\square}="8";
{\ar^>>>>>{\varpi_2} "1";"2"};
{\ar_>>>>>{\varpi_1} "-1";"4"};
{\ar^{p_2} "2";"6"};
{\ar_{p_1} "4";"6"};
\endxy
\]
is a pull-back diagram, $p=p_1\circ\varpi_1=p_2\circ\varpi_2$, and $m_{A_1}\ast m_{A_2}$ is defined by
\[ m_{A_1}\ast m_{A_2}(a_1,a_2)=m_{A_1}(a_1)m_{A_2}(a_2)\ \ ({}^{\forall}(a_1,a_2)\in A_1\times_XA_2). \]
The crossed Burnside ring $\Omega_Q(X)$ is defined to be the Grothendieck ring of this category:
\[ \Omega_Q(X)=K_0(\CrossG) \]
\end{dfn}
Remark that if $Q$ is trivial, then $\Omega_Q(X)$ is nothing other than the ordinary Burnside ring $\Omega(X)$.
As shown in \cite{O-Y2} and \cite{O-Y3}, $\Omega_Q$ has a structure of a Tambara functor. For any $f\in\sG(X,Y)$, the structure morphisms $f^{\ast}$, $f_+$, $f_{\bullet}$ are defined in the following way, which generalizes those for the ordinary Burnside ring functor $\Omega$.
\begin{enumerate}
\item[{\rm (i)}] $f^{\ast}\colon\Omega_Q(Y)\rightarrow\Omega_Q(X)$ is the ring homomorphism induced from
\[ f^{\ast}(B\overset{q}{\rightarrow}Y,m_B)=(A\overset{p}{\rightarrow}X,m_B\circ f\ppr) \quad ({}^{\forall}(B\overset{q}{\rightarrow}Y,m_B)\in \Ob(G\text{-}\mathit{xset}/QY)), \] where
\[
\xy
(-6,6)*+{A}="0";
(6,6)*+{B}="2";
(-6,-6)*+{X}="4";
(6,-6)*+{Y}="6";
(0,0)*+{\square}="8";
{\ar^{f^{\prime}} "0";"2"};
{\ar_{p} "0";"4"};
{\ar^{q} "2";"6"};
{\ar_{f} "4";"6"};
\endxy
\]
is the pull-back.
\item[{\rm (ii)}] $f_+\colon\Omega_Q(X)\rightarrow\Omega_Q(Y)$ is the additive homomorphism induced from
\[ f_+\Am=(A\overset{f\circ p}{\longrightarrow}Y,m_A) \quad ({}^{\forall}\Am\in \Ob(\CrossG)). \]
\item[{\rm (iii)}] To define $f_{\bullet}\colon\Omega_Q(X)\rightarrow\Omega_Q(Y)$, remark that there exists a ring isomorphism
\[ \Omega_Q(X)\overset{\cong}{\longrightarrow}\Omega(X\times Q), \]
which takes $\Am$ to $(A\overset{(p,m_A)}{\longrightarrow}X\times Q)$.
From this, we can import the multiplicative transfers into $\Omega_Q$ from those for the ordinary Burnside ring functor $\Omega$.
Let
\[
\xy
(-22,6)*+{X}="0";
(-22,-6)*+{Y}="2";
(-7,6)*+{X\times Q}="4";
(5,6)*+{}="5";
(18,4.7)*+{X\underset{Y}{\times}\Pi_f(X\times Q)}="6";
(16,2)*+{}="7";
(16,-6)*+{\Pi_f(X\times Q)}="8";
(31,6)*+{}="9";
(-6,0)*+{\mathit{exp}}="10";
{\ar_{f} "0";"2"};
{\ar_{p_X} "4";"0"};
{\ar_>>>>>{e} "5";"4"};
{\ar^>>>>{f\ppr} "7";"8"};
{\ar^{\pi} "8";"2"};
\endxy
\]
be an exponential diagram, and let
\[ \mu_f\colon\Pi_f(X\times Q)\rightarrow Y\times Q \]
be a morphism defined by
\[ \mu_f(y,\sigma)=(y,\prod_{x\in f^{-1}(y)}\sigma(x)) \]
for any $(y,\sigma)\in\Pi_f(X\times Q)$. Remark that $\Pi_f(X\times Q)$ and $e$ is
\[
\Pi_f(X\times Q)=\Set{(y,\sigma)|%
\begin{array}{l}%
y\in Y, \\
\sigma\colon f^{-1}(y)\rightarrow X\times Q \ \, \text{is a map of sets},\\
\ p_X\circ \sigma=\mathrm{id}_{f^{-1}(y)}%
\end{array}},\]
and $e$ is defined by $e(x,(y,\sigma))=\sigma(x)$ for any $(x,(y,\sigma))\in X\underset{Y}{\times}\Pi_f(X\times Q)$.

Using these, define $f_{\bullet}\colon\Omega_Q(X)\rightarrow\Omega_Q(Y)$ by
\begin{eqnarray*}
f_{\bullet}=(\Omega_Q(X)\cong\Omega(X\times Q)%
&\overset{e^{\ast}}{\longrightarrow}&\Omega(X\times_Y\Pi_f(X\times Q))\\
&\overset{f\ppr_{\bullet}}{\longrightarrow}&\Omega(\Pi_f(X\times Q))\\
&\overset{(\mu_f)_+}{\longrightarrow}&\Omega(Y\times Q)\cong\Omega_Q(Y)),
\end{eqnarray*}
where $f\ppr_{\bullet}$ is the multiplicative transfer for the ordinary Burnside ring functor defined in \cite{Tam}.
For any $S\in\Ob(\sG/(X\times_Y\Pi_f(X\times Q)))$, its image $f\ppr_{\bullet}(S)$ is defined to be $\Pi_{f\ppr}(S)$, by the exponential diagram:
\[
\xy
(5,6)*+{}="5";
(18,4.7)*+{X\underset{Y}{\times}\Pi_f(X\times Q)}="6";
(16,2)*+{}="7";
(16,-6)*+{\Pi_f(X\times Q)}="8";
(31,6)*+{}="9";
{\ar_>>>>{f\ppr} "7";"8"};
(39,6)*+{S}="12";
(50,6)*+{T}="14";
(50,-6)*+{\Pi_{f\ppr}(S)}="16";
(34,0)*+{\mathit{exp}}="18";
{\ar_{} "12";"9"};
{\ar_>>>>>{} "14";"12"};
{\ar^>>>>{} "14";"16"};
{\ar^{} "16";"8"};
\endxy
\]
\end{enumerate}

\begin{prop}\label{PropCBR}
For any finite $G$-monoid $Q$, there is a natural isomorphism of Tambara functors
\[ \varphi\colon\mathcal{T}_{\mathcal{P}_Q}\overset{\cong}{\longrightarrow}\Omega_Q.  \]
\end{prop}
\begin{proof}
For each $X\in\Ob(\sG)$, categories $\CrossG$ and $\mathcal{P}_Q\text{-}\sG/X$ in Definition \ref{DefSM} are obviously equivalent, through the natural functor
\begin{eqnarray*}
\CrossG&\rightarrow&\mathcal{P}_Q\text{-}\sG/X\\
\text{object}\colon\qquad\Am&\mapsto&\Am,\\
\text{morphism}\colon\qquad\qquad\qquad\quad\, f&\mapsto& f.
\end{eqnarray*}
Since this functor is compatible with sums and products, it yields a ring isomorphism $\varphi_X\colon\mathcal{T}_{\mathcal{P}_Q}(X)\overset{\cong}{\longrightarrow}\Omega_Q(X)$.
So it remains to show $\varphi=\{\varphi_X\}_{X\in\Ob(\sG)}$ is compatible with the structure morphisms of $\mathcal{T}_{\mathcal{P}_Q}$ and $\Omega_Q$.

Let $f\in\sG(X,Y)$ be any morphism. Obviously $\varphi$ is compatible with $f_+$ and $f^{\ast}$. Thus it suffices to show the compatibility with the multiplicative transfer $f_{\bullet}$.
By the construction of $f_{\bullet}=(\mathcal{T}_{\mathcal{P}_Q})_{\bullet}(f)=(\gamma\mathcal{S}_{\mathcal{P}_Q})_{\bullet}(f)$, it is the only algebraic map (see \cite{D-S} for the definition) which makes the following diagram commutative.
\[
\xy
(-12,6)*+{\mathcal{S}_{\mathcal{P}_Q}(X)}="0";
(12,6)*+{\mathcal{T}_{\mathcal{P}_Q}(X)}="2";
(-12,-6)*+{\mathcal{S}_{\mathcal{P}_Q}(Y)}="4";
(12,-6)*+{\mathcal{T}_{\mathcal{P}_Q}(Y)}="6";
{\ar@{^(->} "0";"2"};
{\ar_{(\mathcal{S}_{\mathcal{P}_Q})_{\bullet}(f)} "0";"4"};
{\ar@{^(->} "4";"6"};
{\ar^{(\mathcal{T}_{\mathcal{P}_Q})_{\bullet}(f)} "2";"6"};
{\ar@{}|\circlearrowright "0";"6"};
\endxy
\]

Since $f_{\bullet}\colon\Omega_Q(X)\rightarrow\Omega_Q(Y)$ is also an algebraic map as is a multiplicative transfer, it suffices to show the commutativity of the following diagram.
\[
\xy
(-10,6)*+{\mathcal{S}_{\mathcal{P}_Q}(X)}="0";
(10,6)*+{\mathcal{T}_{\mathcal{P}_Q}(X)}="2";
(-10,-6)*+{\mathcal{S}_{\mathcal{P}_Q}(Y)}="4";
(10,-6)*+{\mathcal{T}_{\mathcal{P}_Q}(Y)}="6";
(30,6)*+{\Omega_Q(X)}="8";
(30,-6)*+{\Omega_Q(Y)}="10";
{\ar@{^(->} "0";"2"};
{\ar_{(\mathcal{S}_{\mathcal{P}_Q})_{\bullet}(f)} "0";"4"};
{\ar@{^(->} "4";"6"};
{\ar^{\varphi_X}_{\cong} "2";"8"};
{\ar_{\varphi_Y}^{\cong} "6";"10"};
{\ar^{f_{\bullet}} "8";"10"};
\endxy
\]
Take any element $\Am\in\mathcal{S}_{\mathcal{P}_Q}(X)$.
We use the notation in Definition \ref{DefCrossBurn}.
If we let
\[
\xy
(-22,6)*+{X}="0";
(-22,-6)*+{Y}="2";
(-7,6)*+{X\times Q}="4";
(5,6)*+{}="5";
(18,4.7)*+{X\underset{Y}{\times}\Pi_f(X\times Q)}="6";
(16,2)*+{}="7";
(16,-6)*+{\Pi_f(X\times Q)}="8";
(31,6)*+{}="9";
(-6,0)*+{\mathit{exp}}="10";
{\ar_{f} "0";"2"};
{\ar_{p_X} "4";"0"};
{\ar_>>>>>{e} "5";"4"};
{\ar^>>>>{f\ppr} "7";"8"};
{\ar^{\pi} "8";"2"};
(39,6)*+{A\ppr}="12";
(50,6)*+{Z}="14";
(50,-6)*+{\Pi}="16";
(34,0)*+{\mathit{exp}}="18";
{\ar_{q} "12";"9"};
{\ar_>>>>>{\lambda} "14";"12"};
{\ar^>>>>{\rho} "14";"16"};
{\ar^{\varpi} "16";"8"};
\endxy
\]
be two exponential diagrams, where
\[
\xy
(-12,-6)*+{X\times Q}="0";
(-12,6)*+{A}="2";
(4,-6)*+{}="5";
(18,-7)*+{X\underset{Y}{\times}\Pi_f(X\times Q)}="6";
(12,-3.5)*+{}="7";
(12,6)*+{A\ppr}="8";
(0,0)*+{\square}="10";
{\ar_{(p,m_A)} "2";"0"};
{\ar^>>>>>>{e} "5";"0"};
{\ar^>>>>{q} "8";"7"};
{\ar_{e\ppr} "8";"2"};
\endxy
\]
is a pull-back, then by Lemma \ref{LemC}, the following diagram becomes an exponential diagram.
\[
\xy
(-14,6)*+{X}="0";
(-14,-6)*+{Y}="2";
(0,6)*+{A}="4";
(14,6)*+{Z}="6";
(14,-6)*+{\Pi}="8";
(0,0)*+{\mathit{exp}}="10";
{\ar_{f} "0";"2"};
{\ar_{p} "4";"0"};
{\ar_>>>>>{e\ppr\circ\lambda} "6";"4"};
{\ar^>>>>{\rho} "6";"8"};
{\ar^{\pi\circ\varpi} "8";"2"};
\endxy
\]
Thus, composing appropriate isomorphisms, we may assume
\begin{eqnarray*}
\Pi&=&\Pi_f(A)=\Set{(y,\sigma)|%
\begin{array}{l}%
y\in Y, \\
\sigma\colon f^{-1}(y)\rightarrow A \ \, \text{is a map of sets},\\
\ p\circ \sigma=\mathrm{id}_{f^{-1}(y)}%
\end{array}},\\
Z&=&X\times_Y\Pi_f(A).
\end{eqnarray*}

Let $p_1\colon X\times Q\rightarrow Q$ and $p_2\colon Y\times Q\rightarrow Q$ be the projections onto $Q$.
By the definition of $f_{\bullet}\colon\Omega_Q(X)\rightarrow\Omega_Q(Y)$,
we have
\begin{eqnarray*}
f_{\bullet}\circ\varphi_X\Am
&=&(\mu_f)_+f_\bullet\ppr e^{\ast}(A\overset{(p,m_A)}{\longrightarrow}X\times Q)\\
&=&(\mu_f)_+f_\bullet\ppr (A\ppr\overset{q}{\rightarrow}X\times \Pi_f(X\times Q))\\
&=&(\mu_f)_+(\Pi\overset{\varpi}{\rightarrow}\Pi_f(X\times Q))\\
&=&(\Pi\overset{\mu_f\circ\varpi}{\longrightarrow}Y\times Q)\in \Omega(Y\times Q)\ (\cong \Omega_Q(Y)).\\
\end{eqnarray*}
On the other hand,
\begin{eqnarray*}
(\varphi_Y\circ(\mathcal{S}_{\mathcal{P}_Q})_{\bullet})(f)\Am
&=&(\Pi\overset{\pi\circ\varpi}{\longrightarrow}Y,\rho_{\ast}(e\ppr\circ\lambda)^{\ast}(m_A))\\
&=&(\Pi\overset{\pi\circ\varpi}{\longrightarrow}Y,\rho_{\ast}(m_A\circ e\ppr\circ\lambda))\\
&=&(\Pi\overset{\pi\circ\varpi}{\longrightarrow}Y,\rho_{\ast}(p_1\circ e\circ(X\times_Y\varpi)))\ \in\Omega_Q(Y).
\end{eqnarray*}

Thus it remains to show these two elements coincide, under the isomorphism $\Omega(Y\times Q)\cong\Omega_Q(Y)$.
Since $\mathrm{pr}_Y\circ\mu_Y\circ\varpi=\pi\circ\varpi$, it suffices to show
\begin{equation}
p_2\circ\mu_f\circ\varpi=\rho_{\ast}(p_1\circ e\circ(X\times_Y\varpi)).
\label{SuffToShow**}
\end{equation}
Let $(y,\sigma)\in\Pi_f(A)=\Pi$ be any element, where $\sigma\colon f^{-1}(y)\rightarrow A$ is a map of sets. Its image under $\varpi$ can be written as
\[ \varpi(y,\sigma)=(y,\tau), \]
with some map of sets $\tau\colon f^{-1}(y)\rightarrow X\times Q$.
Then we have
\begin{eqnarray*}
p_2\circ\mu_f\circ\varpi(y,\sigma)&=&p_2\circ\mu_f(y,\tau)\\
&=&\prod_{x\in f^{-1}(y)}(p_1(\tau(x))),
\end{eqnarray*}
\begin{eqnarray*}
\rho_{\ast}(p_1\circ e\circ(X\times_Y\varpi))(y,\sigma)&=&\!\prod_{x\in f^{-1}(y)}((p_1\circ e\circ(X\times_Y\varpi))(x,(y,\sigma)))\\
&=&\!\prod_{x\in f^{-1}(y)}((p_1\circ e)(x,(y,\tau))\ =\! \prod_{x\in f^{-1}(y)}(p_1(\tau(x))),
\end{eqnarray*}
and $(\ref{SuffToShow**})$ is satisfied.
\end{proof}

\subsection{Relation with the Witt-Burnside construction}
\label{3-2}

Let $\mathbb{W}_G(R)$ be the Witt-Burnside ring associated to a ring $R$ and a profinite group $G$. For the definition, see \cite{D-S}.
If $G$ is finite, as we assume in this section, then the Witt-Burnside rings are related to Tambara functors as follows.
\begin{fact}[Theorem B, Theorem 15 in \cite{Brun}]\label{FactBrun}
For any finite group $G$, the evaluation at $G/e$
\[ \TamG\rightarrow\GRing\ ; \ T\mapsto T(G/e) \]
has a right adjoint functor $L_G$. Here, $\GRing$ denotes the category of $G$-rings. If $G$ acts trivially on a ring $R$, then for any subgroup $H\le G$, there is an isomorphism
\[ \mathbb{W}_H(R)\cong L_G(R)(G/H). \]
\end{fact}

\begin{rem}
If $T$ is a Tambara functor on $G$, then $T(G/e)$ carries a natural $G$-ring structure induced from $T^{\ast}$ of translations. We denote its $G$-fixed part by $T(G/e)^G$. This gives a functor to $\Ring$
\[ \ev\colon\TamG\rightarrow\Ring\ ;\ T\mapsto T(G/e)^G ,\]
which we call the \lq $G$-invariant evaluation' functor.
\end{rem}

Recall that the functor taking the $G$-fixed part
\[ \GRing\rightarrow\Ring ;\ R\mapsto R^G \]
is left adjoint to the natural functor
\[ \Ring\rightarrow\GRing\ ;\ R\mapsto R\triv, \]
which endows a ring with the trivial $G$-action.
Here, $R\triv$ denotes a ring $R$ with the trivial $G$-action.
Thus by Brun's theorem, we obtain:
\begin{cor}
$G$-invariant evaluation functor $\ev\colon\TamG\rightarrow\Ring$ has a left adjoint functor, which we denote by
\[ \mathbb{W}\colon\Ring\rightarrow\TamG, \]
satisfying
$(\mathbb{W}(R))(G/H)\cong\mathbb{W}_H(R)$
for any ring $R$ and for each $H\le G$.
\end{cor}

Remark that the category of semi-Mackey functors also admits the \lq$G$-invariant evaluation' functor
\[ \ev\colon\SMackG\rightarrow\Mon\ ;\ M\mapsto M(G/e)^G, \]
compatible with that on $\TamG$: 
\begin{equation}
\xy
(-36,6)*+{T=(T^{\ast},T_+,T_{\bullet})}="-10";
(-36,-6)*+{(T^{\ast},T_{\bullet})}="-12";
(-14,6)*+{\TamG}="0";
(14,6)*+{\Ring}="2";
(-14,-6)*+{\SMackG}="4";
(14,-6)*+{\Mon}="6";
{\ar^{\ev} "0";"2"};
{\ar_{\mu\circ\mathcal{U}} "0";"4"};
{\ar^<<<{\text{multiplicative}} "2";"6"};
{\ar^>>>{\text{part}} "2";"6"};
{\ar_{\ev} "4";"6"};
{\ar@{}|\circlearrowright "0";"6"};
{\ar@{|->} "-10";"-12"};
\endxy
\label{diagEV}
\end{equation}

We construct the left adjoint functor of $\ev\colon\SMackG\rightarrow\Mon$ in the following. Remark that (cf. \cite{Bouc}), to give a semi-Mackey functor $M$ on $G$ is equivalent to give a data
\begin{eqnarray*}
&M(G/H)\in\Ob(\Mon)&\\
&r^H_K\in\Mon(M(G/H),M(G/K))&(\text{restriction})\\
&t^H_K\in\Mon(M(G/K),M(G/H))&(\text{transfer})\\
&c_{g,H}\in\Mon(M(G/H),M(G/{}^g\! H))&(\text{conjugation map})
\end{eqnarray*}
for each $K\le H\le G$ and $g\in G$, satisfying
\begin{eqnarray*}
&t^H_K\circ t^K_L=t^H_L,\ r^K_L\circ r^H_K=r^H_L&({}^{\forall}L\le{}^{\forall}K\le{}^{\forall}H\le G)\\
&c_{g_2,{}^{g_1}\! H}\circ c_{g_1,H}=c_{g_2g_1,H}&({}^{\forall}g_1,g_2\in G, {}^{\forall}H\le G)\\
&c_{g,H}\circ t^H_K=t^{{}^g\! H}_{{}^g\! K}\circ c_{g,K}&({}^{\forall}g\in G, {}^{\forall}K\le{}^{\forall}H\le G)\\
&c_{g,K}\circ r^H_K=r^{{}^g\! H}_{{}^g\!K}\circ c_{g,H}&({}^{\forall}g\in G, {}^{\forall}K\le{}^{\forall}H\le G)
\end{eqnarray*}
and the Mackey condition
\begin{equation}
r^H_L\circ t^H_K=\!\!\displaystyle{\sum_{h\in L\backslash H/K}}t^L_{L\cap{}^h\! K}\circ c_{h,L^h\cap K}\circ r^K_{L^h\cap K} \ \ \ ({}^{\forall}K\le{}^{\forall}H\le G).
\label{MackCondBouc}
\end{equation}

In this description, a morphism $\varphi\colon M\rightarrow N$ between semi-Mackey functors $M$ and $N$ corresponds to a family $\varphi=\{\varphi_H\}_{H\le G}$ of monoid homomorphisms
\[ \varphi_H\colon M(G/H)\rightarrow N(G/H) \]
compatible with conjugations, restrictions and transfers.

\begin{dfn}\label{DefConst}
Let $Q$ be a monoid. Define a Mackey functor $\mathcal{L}_Q$ by
\begin{eqnarray*}
\mathcal{L}_Q(G/H)=Q&({}^{\forall} H\le G)\\
r^H_K=[H:K]\colon Q\rightarrow Q&({}^{\forall}K\le {}^{\forall}H\le G)\\
t^H_K=\mathrm{id}_Q\colon Q\rightarrow Q&({}^{\forall}K\le {}^{\forall}H\le G)\\
c_{g,H}=\mathrm{id}_Q\colon Q\rightarrow Q&({}^{\forall}g\in G,{}^{\forall}H\le G)
\end{eqnarray*}
where $[H:K]$ denotes the multiplication by the index $[H:K]$.
\end{dfn}

\begin{rem}
Let $\mathcal{P}_Q$ be the fixed point functor, where $Q$ is regarded as a $G$-monoid with the trivial $G$-action.
Then $\mathcal{L}_Q(G/H)=Q=\mathcal{P}_Q(G/H)$ for each $H\le G$, and $\mathcal{L}_Q$ is the Mackey functor whose restrictions and transfers are reversed from $\mathcal{P}_Q$.
\end{rem}

\begin{claim}\label{ClaimConst}
The correspondence $Q\mapsto\mathcal{L}_Q$ forms a functor
\[ \mathcal{L}\colon\Mon\rightarrow\SMackG, \]
which is left adjoint to $\ev$.
\end{claim}

If Claim \ref{ClaimConst} is proven, we obtain the following.
\begin{thm}\label{Thm3}
For any finite group $G$, there is an isomorphism of functors
\begin{equation}
\mathbb{W}\circ\mathbb{Z}[-]\cong\mathcal{T}\circ\mathcal{L}.
\qquad
\xy
(-12,6)*+{\Mon}="0";
(12,6)*+{\SMackG}="2";
(-12,-6)*+{\Ring}="4";
(12,-6)*+{\TamG}="6";
{\ar^<<<<<<{\mathcal{L}} "0";"2"};
{\ar_{\mathbb{Z}[-]} "0";"4"};
{\ar^{\mathcal{T}} "2";"6"};
{\ar_<<<<<<{\mathbb{W}} "4";"6"};
{\ar@{}|\circlearrowright "0";"6"};
\endxy
\label{diagWT}
\end{equation}
\end{thm}
\begin{proof}
Suppose Claim \ref{ClaimConst} is shown.
Then, each functor in $(\ref{diagWT})$ is left adjoint to the corresponding functor in $(\ref{diagEV})$.
Thus Theorem \ref{Thm3} follows from the commutativity of $(\ref{diagEV})$ and the uniqueness of the left adjoint functor.
\end{proof}
\begin{proof}[Proof of Claim \ref{ClaimConst}]
Let $Q_1$, $Q_2$ be two monoids. To each monoid homomorphism $\theta\colon Q_1\rightarrow Q_2$, we can naturally associate a morphism of semi-Mackey functors $\mathcal{L}_{\theta}\colon\mathcal{L}_{Q_1}\rightarrow\mathcal{L}_{Q_2}$ by
\[ \mathcal{L}_{\theta}(G/H)=\theta\colon\mathcal{L}_{Q_1}(G/H)\rightarrow\mathcal{L}_{Q_2}(G/H)\ \ \ ({}^{\forall}H\le G), \]
and obtain a functor $\mathcal{L}\colon\Mon\rightarrow\SMackG$.

To show the adjointness, let $Q\in\Ob(\Mon)$ and $M\in\Ob(\SMackG)$ be any pair of objects.
It suffices to construct natural maps
\begin{eqnarray*}
\Phi\colon&\Mon(Q,M(G/e)^G)\rightarrow\SMackG(\mathcal{L}_Q,M)\\
\Theta\colon&\SMackG(\mathcal{L}_Q,M)\rightarrow\Mon(Q,M(G/e)^G)
\end{eqnarray*}
which are inverses to each other.

First we define $\Theta$.
Let $\varphi\in\SMackG(\mathcal{L}_Q,M)$ be any morphism.
Since $G$ acts trivially on $\mathcal{L}_Q(G/H)=Q$, the component of $\varphi$ at $G/e$
\[ \varphi_{G/e}\colon\mathcal{L}_Q(G/e)\rightarrow M(G/e) \]
factors through $M(G/e)^G$. This gives a monoid homomorphism
\[ \Theta(\varphi)\colon Q\rightarrow M(G/e)^G. \]

Second, we define $\Phi$.
Let $\theta\in\Mon(Q,M(G/e)^G)$ be any morphism. We define $\Phi(\theta)=\{\Phi(\theta)_H\}_{H\le G}$ by the composition
\[ \Phi(\theta)_H=(\mathcal{L}_Q(G/H)=Q\overset{\theta}{\rightarrow}M(G/e)^G\hookrightarrow M(G/e)\overset{t^H_e}{\rightarrow}M(G/H)), \]
for each $H\le G$.

\begin{claim}\label{ClaimAdj}
$\Phi(\theta)=\{\Phi(\theta)_H\}_{H\le G}\in\SMackG(\mathcal{L}_Q,M)$.
\end{claim}
If Claim \ref{ClaimAdj} is shown, then we can easily show
\begin{eqnarray*}
\Phi\circ\Theta(\varphi)=\varphi&({}^{\forall}\varphi\in\SMackG(\mathcal{L}_Q,M)),\\
\Theta\circ\Phi(\theta)=\theta&({}^{\forall}\theta\in\Mon(Q,M(G/e)^G)),
\end{eqnarray*}
and thus obtain the desired adjoint isomorphism
\[ \Mon(Q,M(G/e)^G)\cong\SMackG(\mathcal{L}_Q,M). \]

\smallskip

Thus it remains to show Claim \ref{ClaimAdj}, namely, to show the compatibility of $\{\Phi(\theta)_H \}_{H\le G}$ with the structure morphisms of semi-Mackey functors $\mathcal{L}_Q$ and $M$.

\begin{enumerate}
\item[{\rm (i)}] compatibility with conjugation maps\\
For any $g\in G$ and $H\le G$, we have $c_{g,H}\circ\Phi(\theta)_H=\Phi(\theta)_{({}^g\! H)}\circ c_{g,H}$. This follows from the following commutative diagram.
Remark $c_{g,e}$ is the identity on $M(G/e)^G$.
\[
\xy
(-36,6)*+{\mathcal{L}_Q(G/H)}="0";
(-20,6)*+{Q}="2";
(0,6)*+{M(G/e)^G}="4";
(22,6)*+{M(G/e)}="6";
(46,6)*+{M(G/H)}="8";
(-36,-6)*+{\mathcal{L}_Q(G/{}^g\! H)}="10";
(-20,-6)*+{Q}="12";
(0,-6)*+{M(G/e)^G}="14";
(22,-6)*+{M(G/e)}="16";
(46,-6)*+{M(G/{}^g\! H)}="18";
{\ar@{=} "0";"2"};
{\ar^<<<<<{\theta} "2";"4"};
{\ar@{^(->} "4";"6"};
{\ar^{t^H_e} "6";"8"};
{\ar@{=} "10";"12"};
{\ar_<<<<<{\theta} "12";"14"};
{\ar@{^(->} "14";"16"};
{\ar_{t^{{}^g\! H}_e} "16";"18"};
{\ar_{c_{g,H}}"0";"10"};
{\ar^{\mathrm{id}}"2";"12"};
{\ar_{\mathrm{id}}"4";"14"};
{\ar_{c_{g,e}} "6";"16"};
{\ar^{c_{g,H}} "8";"18"};
{\ar@{}|\circlearrowright "0";"12"};
{\ar@{}|\circlearrowright "2";"14"};
{\ar@{}|\circlearrowright "4";"16"};
{\ar@{}|\circlearrowright "6";"18"};
\endxy
\]

\item[{\rm (ii)}] compatibility with inductions\\
For any sequence of subgroups $K\le H\le G$, we have $t^H_K\circ\Phi(\theta)_K=\Phi(\theta)_H\circ t^H_K$. This immediately follows from the commutativity of the following diagram.
\[
\xy
(-36,6)*+{\mathcal{L}_Q(G/K)}="0";
(-20,6)*+{Q}="2";
(0,6)*+{M(G/e)^G}="4";
(22,6)*+{M(G/e)}="6";
(46,6)*+{M(G/K)}="8";
(-36,-6)*+{\mathcal{L}_Q(G/H)}="10";
(-20,-6)*+{Q}="12";
(0,-6)*+{M(G/e)^G}="14";
(22,-6)*+{M(G/e)}="16";
(46,-6)*+{M(G/H)}="18";
{\ar@{=} "0";"2"};
{\ar^<<<<<{\theta} "2";"4"};
{\ar@{^(->} "4";"6"};
{\ar^{t^K_e} "6";"8"};
{\ar@{=} "10";"12"};
{\ar_<<<<<{\theta} "12";"14"};
{\ar@{^(->} "14";"16"};
{\ar_{t^H_e} "16";"18"};
{\ar_{t^H_K}"0";"10"};
{\ar^{\mathrm{id}}"2";"12"};
{\ar_{\mathrm{id}} "6";"16"};
{\ar^{t^H_K} "8";"18"};
{\ar@{}|\circlearrowright "0";"12"};
{\ar@{}|\circlearrowright "2";"16"};
{\ar@{}|\circlearrowright "6";"18"};
\endxy
\]

\item[{\rm (iii)}] compatibility with restrictions\\
For any sequence of subgroups $K\le H\le G$, we have $r^H_K\circ\Phi(\theta)_H=\Phi(\theta)_K\circ r^H_K$. This follows from the commutativity of the following diagram.
\[
\xy
(-36,6)*+{\mathcal{L}_Q(G/H)}="0";
(-20,6)*+{Q}="2";
(0,6)*+{M(G/e)^G}="4";
(22,6)*+{M(G/e)}="6";
(46,6)*+{M(G/H)}="8";
(-36,-6)*+{\mathcal{L}_Q(G/K)}="10";
(-20,-6)*+{Q}="12";
(0,-6)*+{M(G/e)^G}="14";
(22,-6)*+{M(G/e)}="16";
(46,-6)*+{M(G/K)}="18";
{\ar@{=} "0";"2"};
{\ar^<<<<<{\theta} "2";"4"};
{\ar@{^(->} "4";"6"};
{\ar^{t^H_e} "6";"8"};
{\ar@{=} "10";"12"};
{\ar_<<<<<{\theta} "12";"14"};
{\ar@{^(->} "14";"16"};
{\ar_{t^K_e} "16";"18"};
{\ar_{r^H_K}"0";"10"};
{\ar^{[H:K]}"2";"12"};
{\ar^{[H:K]}"4";"14"};
{\ar^{r^H_K} "8";"18"};
{\ar@{}|\circlearrowright "0";"12"};
{\ar@{}|\circlearrowright "2";"14"};
{\ar@{}|\circlearrowright "4";"18"};
\endxy
\]
(Remark that, from the Mackey condition, we have
\begin{eqnarray*}
r^H_K\circ t^H_e(x)&=&\displaystyle{\sum_{h\in K\backslash H}}t^K_e\circ c_{h,e}(x)\\
&=&\displaystyle{\sum_{h\in K\backslash H}}t^K_e(x)\ =\ [H:K]\ t^K_e(x)
\end{eqnarray*}
for any $x\in M(G/e)^G$.)
\end{enumerate}
\end{proof}

Theorem \ref{Thm3} also gives a functorial enhancement of Elliott's isomorphism.
For any monoid $Q$ and profinite group $G$, Elliott defined $\mathbf{B}_Q(G)$ to be the Grothendieck ring of the category of $G$-{\it strings} over $Q$. The category of $G$-strings over $Q$ is defined as follows.
\begin{dfn}[\S\S\,\! 2.2 in \cite{Elliott}]\label{DefGStr}
Let $G$ be a profinite group, and let $Q$ be an object in $\Mon$. A $G$-{\it string} over $Q$ is a pair $(A,m_A)$ of
\begin{itemize}
\item[{\rm (a)}] an {\it almost finite} $G$-set $A$,
\item[{\rm (b)}] a map of sets $m_A\colon A\rightarrow Q$, which is constant on $G$-orbits in $A$.
\end{itemize}

The category of $G$-strings is denoted by $G\text{-}\mathit{String}/Q$.
For the general definition of an almost finite $G$-set, see \cite{Elliott}.
Since we are now assuming $G$ is finite, it is nothing other than a finite $G$-set. For each pair of objects $(A,m_A)$ and $(B,m_B)$, their sums and products are defined by
\begin{eqnarray*}
(A,m_A)\amalg(B,m_B)&=&(A\amalg B,m_A\cup m_B),\\
(A,m_A)\times(B,m_B)&=&(A\times B,m_A\ast m_B),
\end{eqnarray*}
where $m_A\ast m_B$ is defined by $(m_A\ast m_B)(a,b)=m_A(a)\cdot m_B(b)$ for any $(a,b)\in A\times B$.
\end{dfn}

In the context of the Witt-Burnside construction, Elliott showed the following.
\begin{fact}[Theorem 1.7 in \cite{Elliott}]\label{FactElliott}
For any profinite group $G$ and any monoid $Q$, there is an isomorphism of rings $\mathbf{B}_Q(G)\cong\mathbb{W}_G(\mathbb{Z}[Q])$.
\end{fact}

When $G$ is finite, our construction of $\mathcal{T}_M$ generalizes the Elliott's ring. In fact, $\mathbf{B}_Q(G)$ is obviously isomorphic to the evaluation of $\mathcal{T}$ of $M=\mathcal{L}_Q$ as follows:
\begin{prop}\label{PropLast}
Let $G$ be a finite group. For any monoid $Q$ and any subgroup $H\le G$, there is a ring isomorphism $\mathcal{T}_{\mathcal{L}_Q}(G/H)\cong\mathbf{B}_Q(H)$.
\end{prop}
\begin{proof}
Since there is an equivalence of categories
\begin{eqnarray*}
\mathcal{L}_Q\text{-}\sG/(G/H)&\rightarrow&H\text{-}\mathit{String}/Q\\
(A\overset{p}{\rightarrow}G/H,m_A)&\mapsto&(p^{-1}(eH),m_A|_{p^{-1}(eH)})
\end{eqnarray*}
preserving sums and products, the corresponding Grothendieck rings become isomorphic:
\[ \mathcal{T}_{\mathcal{L}_Q}(G/H)=K_0(\mathcal{L}_Q\text{-}\sG/(G/H))\cong K_0(H\text{-}\mathit{String}/Q)=\mathbf{B}_Q(H). \]
\end{proof}

Since $\mathcal{T}_{\mathcal{L}_Q}$ is a Tambara functor by Theorem \ref{Thm1}, this isomorphism gives a Tambara functor structure on $\mathbf{B}_Q$.
Thus, applying Proposition \ref{PropLast} to the natural isomorphism
\[ \mathcal{T}\circ\mathcal{L}\cong\mathbb{W}\circ\mathbb{Z}[-] \]
in Theorem \ref{Thm3}, we obtain an isomorphism of Tambara functors
\[ \mathbf{B}_Q\cong\mathcal{T}_{\mathcal{L}_Q}\cong\mathbb{W}\circ\mathbb{Z}[Q]. \]
This functorial isomorphism gives back the Elliott's ring isomorphism in Fact \ref{FactElliott}, by the evaluation at $G/H$ for each $H\le G$.

\subsection{Underlying Green functor structure}
\label{3-3}


\begin{dfn}
Let $\MaddG$ be the category of additive contravariant functors $F$ from $\sG$ to $\Mon$. As in Definition \ref{DefMackFtr}, \lq additive' means that $F$ takes direct sums to direct products.
\end{dfn}

\begin{rem}
By Remark \ref{FinInf}, $\MaddG$ can be identified with the category of additive contravariant functors from $\C$ to $\Mon$, and we have a natural forgetful functor
\[ \mathit{SMack}(G)\rightarrow\MaddG\ \ ;\ \ (M^{\ast},M_{\ast})\mapsto M^{\ast}. \]
\end{rem}

\begin{rem}
By definition, a Mackey functor $(M^{\ast},M_{\ast})$ is a {\it Green functor} if it is equipped with a bifunctorial cross product, namely, if a bilinear map
\[ \mathit{cr}\colon M(X)\times M(Y)\rightarrow M(X\times Y) \]
is given for each $X,Y\in\Ob(\C)$, naturally in $X,Y$ with respect to $M^{\ast}$ and $M_{\ast}$ $($\cite{Bouc}$)$.
For a Tambara functor $(T^{\ast},T_+,T_{\bullet})$, the bilinear maps
\[ \mathit{cr}\colon T(X)\times T(Y)\rightarrow T(X\times Y)\ \ ;\ \ (x,y)\mapsto T^{\ast}(p_X)(x)\cdot T^{\ast}(p_Y)(y), \]
where $p_X\colon X\times Y\rightarrow X$ and $p_Y\colon X\times Y\rightarrow Y$ are the projections, are natural with respect to $T^{\ast}$ and $T_+$. Thus $(T^{\ast},T_+)$ becomes a Green functor. This defines a forgetful functor $\mathit{Tam}(G)\rightarrow\mathit{Green}(G)$,
and we have a commutative diagram of forgetful functors.
\[
\xy
(-24,-8)*+{\mathit{SMack}(G)}="0";
(24,-8)*+{\MaddG}="2";
(-24,8)*+{\mathit{Tam}(G)}="4";
(24,8)*+{\mathit{Green}(G)}="6";
{\ar_{(M^{\ast},M_{\ast})\mapsto M^{\ast}} "0";"2"};
{\ar_{(T^{\ast},T_+,T_{\bullet})\mapsto(T^{\ast},T_{\bullet})} "4";"0"};
{\ar^{(M^{\ast},M_{\ast})\mapsto M^{\ast}} "6";"2"};
{\ar^{(T^{\ast},T_+,T_{\bullet})\mapsto(T^{\ast},T_+)} "4";"6"};
{\ar@{}|\circlearrowright "0";"6"};
\endxy
\]
\end{rem}
\begin{proof}
The proof will be also found in \cite{Tam_Manu}. We briefly demonstrate the naturality of $\mathit{cr}$ with respect to the covariant part. Take any $f\in\sG(X,X\ppr)$ and $g\in\sG(Y,Y\ppr)$. If we let
\[ X\overset{p_X}{\leftarrow}X\times Y\overset{p_Y}{\rightarrow}Y,\ \  %
X\ppr\overset{p_{X\ppr}}{\leftarrow}X\ppr\times Y\ppr\overset{p_{Y\ppr}}{\rightarrow}Y\ppr, \ \ %
X\ppr\overset{\pi_{X\ppr}}{\leftarrow}X\ppr\times Y\overset{\pi_Y}{\rightarrow}Y \]
be the products, then
\[
\xy
(-10,6)*+{X}="0";
(10,6)*+{X\times Y}="2";
(-10,-6)*+{X\ppr}="4";
(10,-6)*+{X\ppr\times Y}="6";
(0,0)*+{\square}="8";
{\ar_{p_X} "2";"0"};
{\ar_{f} "0";"4"};
{\ar^{f\times1_Y} "2";"6"};
{\ar^{\pi_{X\ppr}} "6";"4"};
\endxy
,\qquad
\xy
(-10,6)*+{X\ppr\times Y}="0";
(10,6)*+{Y}="2";
(-10,-6)*+{X\ppr\times Y\ppr}="4";
(10,-6)*+{Y\ppr}="6";
(0,0)*+{\square}="8";
{\ar^{\pi_Y} "0";"2"};
{\ar_{1_{X\ppr}\times g} "0";"4"};
{\ar^{g} "2";"6"};
{\ar_{p_{Y\ppr}} "4";"6"};
\endxy
\]
are pull-backs, and
\[ p_{X\ppr }\circ(1_{X\ppr}\times g)=\pi_{X\ppr},\ \ \ \pi_Y\circ(f\times 1_Y)=p_Y. \]
By the projection formula (\cite{Tam}), for any $x\in T(X)$ and $y\in T(Y)$, we have
\begin{eqnarray*}
(f\times g)_+((p_X^{\ast}x)\cdot(p_Y^{\ast}y))
&=&(1_{X\ppr}\times g)_+(f\times 1_Y)_+((p_X^{\ast}x)\cdot((f\times 1_Y)^{\ast}\pi_Y^{\ast}y)))\\
&=&(1_{X\ppr}\times g)_+(((f\times1_Y)_+p_X^{\ast}x)\cdot(\pi_Y^{\ast}y))\\
&=&(1_{X\ppr}\times g)_+((\pi_{X\ppr}^{\ast}f_{\ast}x)\cdot(\pi_Y^{\ast}y))\\
&=&(1_{X\ppr}\times g)_+(((1_{X\ppr}\times g)^{\ast}p_{X\ppr}^{\ast}f_+x)\cdot(\pi_Y^{\ast}y))\\
&=&(p_{X\ppr}^{\ast}f_+x)\cdot((1_{X\ppr}\times g)_+\pi_Y^{\ast}y))\\
&=&(p_{X\ppr}^{\ast}f_+x)\cdot(p_{Y\ppr}^{\ast}g_+y).
\end{eqnarray*}
\[
\xy
(-16,6)*+{T(X)\times T(Y)}="0";
(16,6)*+{T(X\times Y)}="2";
(-16,-6)*+{T(X\ppr)\times T(Y\ppr)}="4";
(16,-6)*+{T(X\ppr\times Y\ppr)}="6";
{\ar^{\mathit{cr}} "0";"2"};
{\ar_{f_+\times g_+} "0";"4"};
{\ar^{(f\times g)_+} "2";"6"};
{\ar_{\mathit{cr}} "4";"6"};
{\ar@{}|\circlearrowright "0";"6"};
\endxy
\]
\end{proof}

To each $F\in\Ob(\MaddG)$, Jacobson associated a Green functor $\mathcal{A}_F$, which we call the $F$-{\it Burnside ring functor}, in the following way.

\begin{dfn}[\cite{Jacobson}, \cite{N_BrRng}]
Let $G$ be a finite group and let $F$ be any object in $\MaddG$. For each finite $G$-set $X$, define a category $(G,X,F)$ by the following.
\begin{itemize}
\item[{\rm (a)}] An object in $(G,X,F)$ is a pair $\Am$ of a morphism $p$ between finite $G$-sets $A$ and $X$, and an element $m_A\in F(A)$.
\item[{\rm (b)}] A morphism in $(G,X,F)$ from $\Am$ to $(A\ppr\overset{p\ppr}{\rightarrow}X,m_{A\ppr})$ is a morphism of $G$-sets $f\colon A\rightarrow A\ppr$, such that $p\ppr\circ f=p$ and $F(f)(m_{A\ppr})=m_A$.
\end{itemize}
\end{dfn}

\begin{fact}[Theorem 5.2 in \cite{Jacobson}]
The correspondence $F\mapsto\mathcal{A}_F$ gives a functor
\[ \mathcal{A}\colon\MaddG\rightarrow\mathit{Green}(G). \]
\end{fact}

\begin{caution}
In Jacobson's paper \cite{Jacobson}, the contravariant part of a (semi-)Mackey functor is denoted by $M_{\ast}$, and the covariant part is denoted by $M^{\ast}$.
Thus, $M_{\ast}$ in \cite{Jacobson} is our $M^{\ast}$, and $M^{\ast}$ in \cite{Jacobson} is our $M_{\ast}$.)

Since we are assuming $\Mon$ is the category of commutative monoids, the resulting Green functor $\mathcal{A}_F$ becomes commutative.
\end{caution}

\smallskip

Obviously, if $M=(M^{\ast},M_{\ast})$ is a semi-Mackey functor and if $X$ is a finite $G$-set, then our category $M\text{-}\sG/X$ agrees with $(G,X,M^{\ast})$.
Thus, under the identification in Remark \ref{FinInf}, we can show that there exists an isomorphism of Green functors $((\mathcal{T}_M)^{\ast},(\mathcal{T}_M)_{\ast})\cong\mathcal{A}_{(M^{\ast})}$. In fact, this gives a natural isomorphism of functors:
\[
\xy
(-20,6)*+{\mathit{SMack}(G)}="0";
(20,6)*+{\MaddG}="2";
(-20,-6)*+{\mathit{Tam}(G)}="4";
(20,-6)*+{\mathit{Green}(G)}="6";
{\ar^{(M^{\ast},M_{\ast})\mapsto M^{\ast}} "0";"2"};
{\ar_{\mathcal{T}} "0";"4"};
{\ar^{\mathcal{A}} "2";"6"};
{\ar_{(T^{\ast},T_+,T_{\bullet})\mapsto(T^{\ast},T_+)} "4";"6"};
{\ar@{}|\circlearrowright "0";"6"};
\endxy
\]

In this view, we can restate Theorem \ref{Thm1} as follows.

\begin{thmd}
Let $G$ be a finite group and $F$ be an object in $\MaddG$. 
If $F$ is moreover a semi-Mackey functor, namely, if there exists a semi-Mackey functor with the contravariant part $F$, then $\mathcal{A}_F$ has a structure of a Tambara functor.
\end{thmd}

As an example, Theorem \ref{Thm1}$\ppr$ gives a Tambara functor structure on the Brauer ring functor in \cite{Jacobson}. Remark that the Brauer ring functor is defined as $\mathcal{A}_{\widetilde{\Br}}$, where $\widetilde{\Br}$ is the additive contravariant functor defined by the composition
\[ \widetilde{\Br}=(\sG\overset{\mathcal{P}_E}{\rightarrow}\Ring\overset{\Br}{\rightarrow}\Ab\hookrightarrow\Mon), \]
where $E/D$ is a finite Galois extension of fields with Galois group $G$.
\begin{cor}\label{CorBr}
Let $E/D$ be a finite Galois extension of fields with Galois group $G$. There is a Tambara functor $\mathcal{T}_{\widetilde{\Br}}$ on $G$, whose underlying Green functor agrees with $\mathcal{A}_{\widetilde{\Br}}$.
\end{cor}
\begin{proof}
This immediately follows from that $\widetilde{\Br}$ is in fact a Mackey functor.
\end{proof}

\begin{rem}
More generally, if $E/D$ is a finite Galois extension of {\it rings} with Galois group $G$, then $\widetilde{\Br}\in\Ob(\MaddG)$ can be defined in a similar way, and the same result as Corollary \ref{CorBr} holds. This follows from the fact that $\widetilde{\Br}$ becomes a Mackey functor by Proposition 1 in \cite{Ford}.
\end{rem}

Combining Theorem \ref{Thm1}$\ppr$ with Theorem 3.13 in \cite{N_BrRng}, the underlying Green functor structure is also written by using Boltje's $(-_+)$-construction.
In \cite{Boltje}, Boltje defined the functor $(-)_+$ from the category $\ResalgG$ of {\it algebra restriction functors} to the category of Green functors on $G$
\[ (-)_+\colon\ResalgG\rightarrow\GreenG. \]
In \cite{N_BrRng}, we constructed a functor
\[ \mathcal{R}\colon\RaddG\rightarrow\ResalgG\ ;\ F\mapsto\mathcal{R}_F \]
which satisfies $\mathcal{R}_F(H)=F(G/H)$ for any $F\in\Ob(\RaddG)$ and $H\le G$.

Let $\RaddG$ be the category of additive contravariant functors from $\sG$ to $\Ring$. The pair of adjoint functors
\begin{eqnarray*}
\Ring\rightarrow\Mon,&\text{forgetful (multiplicative part)}\\
\mathbb{Z}[-]\colon\Mon\rightarrow\Ring&
\end{eqnarray*}
yields another adjoint pair
\begin{eqnarray*}
\RaddG\rightarrow\MaddG,&\text{forgetful}\\
\mathbb{Z}[-]\colon\MaddG\rightarrow\RaddG&
\end{eqnarray*}
by the composition onto the codomains.
\begin{caution}
Although monoids and rings in \cite{N_BrRng} were not assumed to be commutative, the above adjoint property still holds.
\end{caution}

\begin{rem}
In \cite{N_BrRng}, $\RaddG$ is valued in the category of {\it not necessarily commutative} rings, and therefore $\RaddG$ gave an equivalence of categories in that case (Proposition 3.12 in \cite{N_BrRng}).
At any rate, restricting Proposition 3.12 in \cite{N_BrRng} to our commutative case, we obtain an isomorphism of functors
\[ (\mathcal{R}\circ\mathbb{Z}[-])_+\cong\mathcal{A}. \]
\[
\xy
(-12,6)*+{\MaddG}="0";
(12,6)*+{\GreenG}="2";
(-12,-6)*+{\RaddG}="4";
(12,-6)*+{\ResalgG}="6";
{\ar^{\mathcal{A}} "0";"2"};
{\ar_{\mathbb{Z}[-]} "0";"4"};
{\ar_{(-)_+} "6";"2"};
{\ar_{\mathcal{R}} "4";"6"};
{\ar@{}|\circlearrowright "0";"6"};
\endxy
\]
Namely, for each $F\in\Ob(\MaddG)$, we have a natural isomorphism of Green functors $(\mathcal{R}_{\mathbb{Z}[F]})_+\cong\mathcal{A}_F$.
Here, $\mathbb{Z}[F]\in\Ob(\RaddG)$ is the abbreviation for the composition
\[ \sG\overset{F}{\longrightarrow}\Mon\overset{\mathbb{Z}[-]}{\longrightarrow}\Ring. \]
\end{rem}

\begin{cor}\label{CorPlus}
For any $M=(M^{\ast},M_{\ast})\in\Ob(\SMackG)$, there is a natural isomorphism of Green functors $\mathcal{T}_M\cong(\mathcal{R}_{\mathbb{Z}[M^{\ast}]})_+$.
\end{cor}

\begin{rem}
Using the definition of $(-)_+$ by \cite{Boltje}, we can also determine the ring structure of $\mathcal{T}_M(G/H)$ for each $H\le G$. For the detail, see section 3 in \cite{N_BrRng}.
\end{rem}





\begin{thebibliography}{9}                                                      
\bibitem{B-B}Bley, W.; Boltje, R.: \emph{Cohomological Mackey functors in number theory}. J. Number Theory \textbf{105} (2004) 1--37.

\bibitem{Boltje}Boltje, R.: \emph{Mackey functors and related structures in representation theory and number theory}, Habilitation-Thesis, Universit\"{a}t Augsburg (1995).

\bibitem {Bouc}S. Bouc.: \emph{Green functors and $G$-sets}, Lecture Notes in Mathematics, 1671, Springer-Verlag, Berlin (1977).

\bibitem{Brun} Brun, M.: \emph{Witt vectors and Tambara functors}. Adv. in Math. \textbf{193} (2005) 233--256.

\bibitem{D-S}Dress, A.W.M.; Siebeneicher, C.: \emph{The Burnside ring of profinite groups and the Witt vector construction}. Adv. in Math. \textbf{70} (1988), no. 1, 87--132.

\bibitem{Elliott} Elliott, J.: \emph{Constructing Witt-Burnside rings}. Adv. in Math. \textbf{203} (2006) 319--363.

\bibitem{Ford} Ford, T.J.: \emph{Hecke actions on Brauer groups}. J. Pure Appl. Algebra \textbf{33} (1984) no. 1, 11--17. 

\bibitem{Jacobson} Jacobson, E.T.: \emph{The Brauer ring of a field}. Illinois J. Math. \textbf{30} (1986), 479--510.

18 (1976), 273-278.


\bibitem{N_BrRng} Nakaoka, H.: \emph{Structure of the Brauer ring of a field extension}. Illionis J. Math. \textbf{52} (2008), no. 1, 261--277.

\bibitem{N_Tam} Nakaoka, H.: \emph{Tambara functors on profinite groups and generalized Burnside functors}. Comm. Algebra \textbf{37} (2009) no. 4, 3095--3151.

\bibitem{O-Y2} Oda, F; Yoshida, T: \emph{Crossed Burnside rings II: The Dress construction of a Green functor}. J. of Alg. \textbf{282} (2004) 58--82.

\bibitem{O-Y3} Oda, F; Yoshida, T: \emph{Crossed Burnside rings III: The Dress construction of a Tambara functor}, preprint.

\bibitem{Tam} Tambara, D.: \emph{On multiplicative transfer}. Comm. Algebra \textbf{21} (1993), no. 4, 1393--1420.

\bibitem{Tam_Manu} Tambara, D.: \emph{Multiplicative transfer and Mackey functors}. manuscript.

\bibitem{Yoshida} Yoshida, T.: \emph{Polynomial rings with coefficients in Tambara functors}. (Japanese) S\=urikaisekikenky\=usho K\=oky\=uroku No. 1466 (2006), 21--34.
\end{thebibliography}
\end{document}